\documentclass{amsart}
\usepackage{latexsym,amsmath, amscd, amssymb, amsthm,  euscript, mathrsfs, hyperref}
\usepackage{amssymb, paralist, xspace, graphicx, url, amscd, euscript, mathrsfs, stmaryrd}
\usepackage[dvipsnames]{xcolor}
\usepackage[all]{xy}
\usepackage{tikz}
\usepackage{tikz-cd} 
\usetikzlibrary{arrows.meta}
\usetikzlibrary{arrows}
\usetikzlibrary{cd}
\usetikzlibrary{matrix,arrows,decorations.pathmorphing}
\usepackage{makecell}

\usepackage[english]{babel}

\usepackage[letterpaper,top=2cm,bottom=2cm,left=3cm,right=3cm,marginparwidth=1.75cm]{geometry}

\usepackage{amsmath}
\usepackage{graphicx}
\usepackage[all]{xy}
\usepackage[utf8]{inputenc}
\usepackage{dsfont}
\usepackage{amsfonts}
\usepackage{mathrsfs}
\usepackage{mathtools}

\newtheorem{theorem}[subsection]{Theorem}

\newtheorem{lemma}[subsection]{Lemma}
\newtheorem{proposition}[subsection]{Proposition}
\newtheorem{corollary}[subsection]{Corollary}
\theoremstyle{definition}

\newtheorem{example}[subsection]{Example}
\newtheorem{definition}[subsection]{Definition}

\newtheorem{remark}[subsection]{Remark}

\newcommand{\GL}{\mathrm{GL}}
\newcommand{\SL}{\mathrm{SL}}
\newcommand{\PGL}{\mathrm{PGL}}

\newcommand{\Pic}{\mathrm{Pic}}
\newcommand{\Br}{\mathrm{Br}}
\newcommand{\Spec}{\mathrm{Spec} \ }
\newcommand{\tors}{\mathrm{tors}}
\newcommand{\et}{\mathrm{\acute{e}t}}
\newcommand{\uSpec}{\underline{\mathrm{Spec}}}
\newcommand{\Sym}{\mathrm{Sym}}

\newcommand{\Aut}{\mathrm{Aut}}
\newcommand{\uIsom}{\underline{\mathrm{Isom}}}

\newcommand{\Sch}{\mathrm{Sch}}

\newcommand{\pr}{\mathrm{pr}}
\newcommand{\Tot}{\mathrm{Tot}}
\newcommand{\id}{\mathrm{id}}
\newcommand{\op}{\mathrm{op}}
\newcommand{\pre}{\mathrm{pre}}
\newcommand{\Set}{\mathrm{Sets}}

\newcommand{\Hom}{\mathrm{Hom}}

\newcommand{\ms}{\mathscr}

\newcommand{\N}{\mathbb{N}}
\newcommand{\Z}{\mathbb{Z}}

\newcommand{\G}{\mathbb{G}}
\newcommand{\PP}{\mathbb{P}}
\newcommand{\bbA}{\mathbb{A}}

\counterwithin{equation}{subsection}

\title{The Brauer Group of $BG$ and Gerbe Structures of Moduli Spaces }
\author{Rose Lopez}

\begin{document}
\maketitle

\begin{abstract}
    We study the $\mu _N$-gerbe of curves of genus $g$ with an order $N$ automorphism, and explore what corresponding $H^2$-cohomology classes the components of this stack can have. In particular, we look at curves whose quotients by the order $N$ automorphism are genus 0, and completely determine the Brauer classes of these gerbes. The key technical input is the calculation of the Brauer group of $BG$, for $G$ a smooth connected semisimple linear algebraic group.
\end{abstract}

\section{Introduction}
The Brauer group of a field $k$ is an abelian group which classifies Morita equivalence classes of central simple algebras over $k$ with multiplication given by tensor product. The Brauer group has been generalized to schemes and established as a cohomological invariant of a scheme by Grothendieck in \cite{Gr1, Gr2, Gr3}. Brauer groups have been extended to stacks more recently, and various work has been done computing Brauer groups of stacks, for example, in  \cite{AM,Shin, iyer2021brauergroupsalgebraicstacks,achenjang2024brauergroupmathscry02, achenjang2025brauergroupstamestacks, bishop2025brauergroupstamestacky}. 

In \cite[Section 5]{achenjang2025brauergroupstamestacks}, the Leray spectral sequence is used to compute the Brauer group of $BG$ for a finite linearly reductive group $G$, and in \cite[Section 6]{loughran2025mallesconjecturebrauergroups}, the authors calculate the Brauer group of $BG$ for a finite \'etale group scheme $G$. We calculate the (cohomological) Brauer group of $BG$ for more general $G$ in the following theorem.

\begin{theorem} \label{thm:Br(BG) (1)}
Let $G$ be a smooth connected semisimple linear algebraic group over an algebraically closed field $k$. Let $\widetilde{G}$ be the universal cover of $G$, and $B$ the fundamental group of $G$, which fit into an exact sequence $1\to B\to \widetilde{G}\to G\to 1$. Let $X(B):=\Hom(B,\G_m)$ be the character group of $B$. Then
\[ H^i(BG, \G_m)=  \left\{
\begin{array}{ll}
      k^* & i=0 \\
      0 & i=1 \\
      X(B) & i=2. \\
\end{array} 
\right. \]

For $i=2$, an isomorphism $X(B)\to H^2(BG,\G_m)$ is given by $(\chi:B\to \G_m)\mapsto \chi_*(B\widetilde{G})$, where $\chi_*$ is the pushforward map
$\chi_*:H^2(BG,B)\to H^2(BG,\G_m)$.
\end{theorem}

With the understanding of the Brauer group of $BG$, we may also understand the classes of Brauer-Severi varieties over $BG$, which we call \textit{Brauer-Severi stacks} (see Section \ref{sec: Br-Sev stacks}). Classically, a Brauer-Severi variety of dimension $n$ over a field $k$ is a variety $X$ over a field $k$ that is isomorphic to $\PP^n_{\overline{k}}$ after base change to an algebraic closure of $k$. Its obstruction to being isomorphic to $\PP^n_k$ itself is the existence or non-existence of a line bundle of degree 1 on $X$, which would then be isomorphic to $\ms O_{\PP^n_k}(1)$. More generally, a Brauer-Severi stack of relative dimension $n$ is a smooth proper morphism $\ms X\to S$ which is \'etale locally isomorphic to $\PP^n_S$. The existence of a line bundle of degree 1 on $\ms X$ no longer guarantees that $\ms X$ is globally isomorphic to $\PP^n_S$, but rather only guarantees the existence of a rank $n+1$ vector bundle $V$ on $S$ such that $\ms X\cong \PP V$.

There is an equivalence of categories between Brauer-Severi varieties of relative dimension $n$ over $S$ and $\PGL_{n+1}$-torsors over $S$. The image $\delta([X])\in H^2(S,\G_m)$ of the class of $X$ in $H^1(S,\PGL_{n+1})$ via the short exact sequence $1\to \G_m\to \GL_{n+1} \to \PGL_{n+1}\to 1$ is equal to the class of the $\G_m$-gerbe of degree 1 line bundles on $X$, which we call the \textit{Brauer class of the Brauer-Severi variety $X$}. It is trivial if and only if there is a degree 1 line bundle on $X$. 

We study Brauer-Severi stacks of the form $[\PP V/G]\to BG$, and prove the following theorem which calculates thier Brauer classes. 

\begin{theorem} \label{thm: chi(BG)=ms G (1)}
    Let $V$ be a representation of $\widetilde{G}$, acting via $\rho: \widetilde{G}\to \GL(V)$ such that $B\subseteq \widetilde{G}$ acts by a character $\chi:B\to \G_m$. Then we get an induced action of $G$ on $\PP V$, and $[\PP V/G]\to BG$ is a Brauer-Severi stack over $BG$. Let $\ms G$ be the $\G_m$-gerbe of degree 1 line bundles on $[\PP V/G]$ over $BG$. Then $\chi^{-1}_*(B\widetilde{G})=\ms G$, and thus the class of $\ms G$, which is the Brauer class of $[\PP V/G]\to BG$ in $H^2(BG, \G_m)=X(B)$, is given by $\chi^{-1}$.
\end{theorem}

The initial motivation for studying the Brauer group and Brauer-Severi stacks on $BG$ was to understand the gerbe structure of the stack of genus $g$ curves with an automorphism of order $N$, which we now explain.

Let $\ms M_g$ be the moduli stack of smooth curves of genus $g>1$. The inertia stack $I(\ms M_g)$ is the stack with objects pairs $(X\in \ms M_g, \alpha\in \Aut(X))$ and morphisms are isomorphisms $\phi:X'\to X$ such that $\phi\circ \alpha'=\alpha\circ \phi$. The inertia $I(\ms M_g)$ is a disjoint union
$$I(\ms M_g)=\coprod_{N\in \N>0} I_N(\ms M_g),$$
such that for $(X,\alpha)\in I_N(\ms M_g)$, $\alpha$ has order $N$. We call the substacks $I_N(\ms M_g)$ the {\it $N$th twisted sectors} of the inertia stack.

Since every object of $I_N(\ms M_g)$ has an order $N$ automorphism, this stack is a $\mu_N$-gerbe, and corresponds to some class $[I_N(\ms M_g)]\in H^2(\mu_N)$. We look at the components of $I_N(\ms M_g)$ consisting of curves whose quotients by the order $N$ automorphism $\alpha$ has genus 0 and completely determine the \textit{Brauer classes} of these gerbes. For a $\mu_N$-gerbe $\ms G\to S$ with class $[\ms G]\in H^2(S,\mu_n)$, its pushout to $\G_m$ under the inclusion $\mu_N\hookrightarrow\G_m$ is the $\G_m$-gerbe $\ms G^{\wedge \G_m}\to S$ and its \textit{Brauer class} is given by $[\ms G^{\wedge \G_m}]\in H^2(S,\G_m)$.

Given an object $(X,\alpha)$ of $I_N(\ms M_g)$, we may take the quotient of $X$ by the action of $\mu_N$ induced by $\alpha$ to get a Galois cover with group $\mu_N$, $\pi:X\to C=X/\mu_N$. The quotient curve $C$ can take many possible genera with varying ramification divisors. Following \cite{Pagani_2013}, using ideas from \cite{Pardini1991}, to specify which genera and degrees of ramification divisors are allowed on $C$, we define a {\it $g$-admissible datum} to be an $N+1$-tuple $A=(g', N, d_1,\ldots,d_{N-1})$ such that $g'\leq g$ and $N>1$, which satisfies Riemann-Hurwitz
\begin{equation} \label{eq:Riemann-Hurwitz}
    2g-2=N(2g'-2)+\sum_i d_i(N-\gcd(i,N))
\end{equation}
and the structural equation of abelian covers
\begin{equation}
\sum_i id_i=0\pmod N.
\end{equation}

Define the stack $\ms M_A$ to be the stack with objects 
\[\ms M_A(S)=\{(f:C\to S, D_1, \ldots D_{N-1}, \ms L, \phi)\}\]
where $f:C\to S$ is a smooth curve of genus $g'$, the $D_i$ are disjoint divisors of degree $d_i$, that is, sections of $(\Sym_S^{d_i} C\setminus \Delta_{d_i})\to S$, where $\Delta_{d_i}$ is the big diagonal, $\ms L$ is a line bundle on $C$ and $\phi: \ms L^{\otimes N}\to \ms O_{C} (\sum iD_i)$ is an isomorphism. Morphisms will be defined in Section \ref{section: The moduli spaces}.

Every object of $\ms M_A$ corresponds to a $\mu_N$-cover $X\to C$ of a smooth genus $g$ curve over $C$, ramified over the $D_i$. Conversely, every $\mu_N$-cover with the covering curve having genus $g$ corresponds to some object of $\ms M_A$ for some $g$-admissible datum $A$. We roughly explain this correspondence in Section \ref{section: The moduli spaces}. Let $\ms M_A'$ be the substack of $\ms M_A$ such that the corresponding $X$ is connected. The following theorem relates $I_N(\ms M_g)$ to the $\ms M_A'$.

\begin{theorem} \cite[2.2, Corollary 1]{Pagani_2013}, \cite[Theorem 2.1, Proposition 2.1]{Pardini1991} 
    Fix $g,N>1$. Then 
    $I_N(\ms M_g)=\coprod_{A} \ms M_A'$, where the disjoint union ranges over all $g$-admissible $A=(g', N, d_1, \ldots, d_{N-1})$. 
\end{theorem} 

When $g'=0$, either $\ms M_A'=\ms M_A$ or $\ms M_A'=\emptyset$ \cite[2.2, Remark 3]{Pagani_2013}.  Since we only study the case when $g'=0$ in this paper, it suffices to study the $\ms M_A$ as defined above, and then only a subset of such $\ms M_A$ will correspond to components of $I_N(\ms M_g)$ To study the $\ms M_A$, we define a moduli stack $\overline{\ms M_A}$ that contains $\ms M_A$ as an open substack, namely, we allow the divisors to coincide. 

When $g'=0$, $\overline{\ms M_A}$, can be realized as a $\mu_N$-gerbe over a stack $\overline{\ms N_A}$, which is a composition of Brauer-Severi stacks over the moduli space of genus 0 curves, $B\PGL_2$. Using Theorems \ref{thm:Br(BG) (1)} and \ref{thm: chi(BG)=ms G (1)}, and a stacky version of a theorem of Gabber \cite[Theorem 2, page 193]{Gabber}, which relates the Brauer group of a Brauer-Severi variety to that of its base, we compute the Brauer group of $\overline{\ms N_A}$ and calculate the class of $\overline{\ms M_A}$ in the Brauer group of $\overline{\ms N_A}$:

\begin{theorem} \label{thm:N_A}
    Let $N, d_1,\ldots, d_{N-1}\in \Z$ such that $N|d:=\sum id_i$ and 
    \[g:=\dfrac{1}{2}\left(-2N+\sum_i id_i(N-\gcd(i,N))+2\right)\] 
    is an integer greater than 1. Let $A=(0, N, d_1, \ldots, d_{N-1})$. If all $d_i$ are even, $H^2(\overline{\ms N_A},\G_m)=\Z/2\Z$, and if any $d_i$ is odd, $H^2(\overline{\ms N_A},\G_m)=0$. 
\end{theorem}

\begin{theorem} \label{thm:M_A}
    Let $A=(0, N, d_1, \ldots, d_{N-1})$ be as above. Then $[\overline{\ms M_A}^{\wedge \G_m}]=[\gamma^*\ms G_{d/N}]\in H^2(\overline{\ms N_A},\G_m)$, where $\ms G_{d/N}$ is the $\G_m$-gerbe over $B\PGL_2$ of degree $d/N$ line bundles on genus 0 curves, and $\gamma:\overline{\ms N_A}\to B\PGL_2$ is the natural forgetful map. In particular, if all $d_i$ are even and $d/N$ is odd, then $\overline{\ms M_A}^{\wedge \G_m}$ is nontrivial of order 2, otherwise, $\overline{\ms M_A}^{\wedge \G_m}$ is trivial. Moreover, since the restriction map on Brauer groups is injective, the same holds for $\ms M_A^{\wedge \G_m}$ where the divisors are disjoint. 
\end{theorem}

\subsection{Acknowledgments} The author thanks her advisor, Martin Olsson, for many helpful conversations, and Hannah Larson for making her aware of \cite{Pagani_2013} and the moduli spaces $\ms M_A$. The author was partially supported by the Simons Collaboration on Perfection in Algebra, Geometry, and Topology.

\section{Brauer Group of $BG$} \label{section: Br(BG)}

Let $G$ a smooth connected semisimple linear algebraic group over an algebraically closed field $k$. Let $\widetilde{G}$ be the universal cover of $G$, and $B$ the fundamental group of $G$, which exist by \cite[Corollary 4.6]{FOSSUM1973269}. Then $\widetilde{G}$ is also a smooth connected semisimple linear algebraic group, $B$ is finite and diagonalizable, thus a finite abelian group, and we have an exact sequence
\[1\to B\to \widetilde{G}\to G\to 1,\]
with the image of $B$ in the center of $\widetilde{G}$. 

Let  
\[X(-): (\text{algebraic groups}/k)\to \Set \]
be $H\mapsto \Hom(H,\G_m)$, where $\Hom$ refers to homomorphisms of algebraic groups.

We may view $k^*\subseteq H^0(H, \ms O_H^*)=H^0(H,\G_m)$ as constant functions, and we have a surjection $H^0(H,\G_m)\to X(H)$, by $f\mapsto f/f(e)$, which gives us a split exact sequence 
\[0\to k^*\to H^0(H,\G_m)\to X(H)\to 0,\]
and thus $H^0(H,\G_m)=k^*\oplus X(H)$ (\cite[Theorem 3]{Rosenlicht}, \cite[Lemma 6.5 (iii)]{Sansuc1981GroupeDB}, \cite[Corollary 2.2]{FOSSUM1973269}). Since $G$ is semisimple, $X(G)=0$.

We have $\widetilde{G}\to G$ a $B$-torsor, and $B\widetilde{G}\to BG$ a $B$-gerbe. Given a character $\chi: B\to \G_m$ of $B$, we get pushforward maps 
$$\chi_*: H^1(G, B)\to H^1(G, \G_m), \ \ \chi_*:H^2(BG,B)\to H^2(BG,\G_m).$$ Thus we get morphisms
$$X(B)\to H^1(G,\G_m)=\Pic(G),\ \ X(B)\to H^2(BG,\G_m)=\Br(BG),$$

given by $(\chi:B\to \G_m)\mapsto \chi_*(\widetilde{G})$ and $(\chi:B\to \G_m)\mapsto \chi_*(B\widetilde{G})$, respectively.

It follows from \cite[Corollary 6.11, Remark 6.11.3]{Sansuc1981GroupeDB} that the first map is an isomorphism, also using the fact that $X(\widetilde{G})=0$ and $\Pic(\widetilde{G})=0$. That the second map is an isomorphism is the statement of Theorem \ref{thm:Br(BG) (1)}, rewritten below for the reader's convenience.

\begin{theorem} \label{thm: Br(BG) main} Let $k$ be an algebraically closed field, $G$ a smooth connected semisimple linear algebraic group over $k$, and $\widetilde{G}$ and $B$ as above. Then
\[ H^i(BG, \G_m)=  \left\{
\begin{array}{ll}
      k^* & i=0 \\
      0 & i=1 \\
      \widehat{B} & i=2. \\
\end{array} 
\right. \]
For $i=2$, an isomorphism $\widehat{B}\to H^2(BG,\G_m)$ is given by $(\chi:B\to \G_m)\mapsto \chi_*(B\widetilde{G})$, where $\chi_*$ is the pushforward map
$\chi_*:H^2(BG,B)\to H^2(BG,\G_m)$.
\end{theorem}

\begin{proof}
We use the spectral sequence $E_1^{p,q}=H^q(U_p,\G_m)\Rightarrow H^{p+q}(BG,\G_m)$ of the standard simplicial cover $U_\bullet\to BG$ (\cite[5.2.1.1]{Deligne1974HodgeIII}) coming from $\Spec k\to BG$.

Here, $U_0=\Spec k$, and $U_p=U_{p-1}\times_{BG} \Spec k$ for all $p\geq 1$. We have $p+1$ projection maps $U_p\to U_{p-1},$
each of which project onto all but one copy of $\Spec k$. For all $p\geq 1$, $U_p\cong G^p$. In particular, $U_1=\Spec k\times_{BG} \Spec k \cong G$ and $U_2\cong G\times G$. The two projections 
\[q_1, q_2: \Spec k \times_{BG} \Spec k\to \Spec k\]
are both given by the structure map $\pi: G\to \Spec k$. The three projections
\[p_{23}, p_{13}, p_{12}: \Spec k \times_{BG} \Spec k\times_{BG} \Spec k \to \Spec k\times_{BG}\Spec k\]
are given by 
\[p_2, m, p_1: G\times G\to G,\]
respectively. 
The simplicial cover $U_\bullet \to BG$ can then be written
\[U_\bullet: 
\begin{tikzcd}
\cdots 
    \arrow[r, shift left=3]
   \arrow[r, shift left=1]
    \arrow[r,shift right=1]
    \arrow[r, shift right=3] &
G\times G 
    \arrow[r,shift left=2]
    \arrow[r]
    \arrow[r, shift right=2]
& G   
    \arrow[r, shift left=1]
    \arrow[r,shift right=1]
& \Spec k.    
\end{tikzcd} 
\]
We may apply the functors $H^q(-,\G_m)$ to $U_\bullet$ to get 
\[H^q(U_\bullet ,\G_m):H^q(\Spec k, \G_m)\xrightarrow{d^{0,q}}H^q(G,\G_m) \xrightarrow{d^{1,q}} H^q(G\times G,\G_m)\xrightarrow{} \cdots \]
where the differentials $d^{p,q}$ are given by the alternating sums of the maps induced by the projection maps $U_{p+1}\to U_p$. The complexes $H^q(U_\bullet ,\G_m)$ make up the rows of the first page of the spectral sequence.
   
We write out some low-degree terms on the $E_1$ page:

        \begin{center}
\begin{tikzpicture}[scale=1.5, every node/.style={font=\small}]
  % Axes
  \draw[->] (-0.5, 0) -- (6.25, 0) node[anchor=west] {$p$};
  \draw[->] (0, -0.5) -- (0, 2.5) node[anchor=south] {$q$};

  % Nodes
  \node at (0.25,0.25) {$k^*$};
  \node at (1.75,0.25) {$H^0(G,\G_m)$};
  \node at (3.75, 0.25) {$H^0(G\times G, \G_m)$};
  \node at (5.75, 0.25) {$\cdots$};
  \node at (0.25,1.25) {$0$};
  \node at (1.75,1.25) {$\Pic(G)$};
  \node at (3.75, 1.25) {$\Pic(G\times G)$};
  \node at (5.75, 1.25) {$\cdots$};
  \node at (0.25, 2.25) {0};
  \node at (1.75, 2.25) {$\cdots$};

  % Arrow (differential)
  \draw[->] (0.45,0.25) -- (1.1,0.25);
  \draw[->] (2.35,0.25) -- (2.9, 0.25);
  \draw[->] (4.6, 0.25) -- (5.4, 0.25);
  \draw[->] (0.45, 1.25) -- (1.1,1.25);
  \draw[->] (2.35, 1.25) -- ( 2.9, 1.25);
  \draw[->] (4.6, 1.25) -- (5.4, 1.25);
  \draw[->] (0.45, 2.25) -- (1.1, 2.25);
  
\end{tikzpicture}
\end{center}
 
    Note that since $k$ is algebraically closed, $\Br(k)=\Pic(k)=0$. Note also that for all $q$, $d^{0,q}$ is the alternating sum of two copies of 
    $\pi^*: H^q(\Spec k, \G_m)\to H^q(G,\G_m)$, and thus $d^{0,q}=0$. 

    We first compute the cohomology of the bottom line, $H^p(E_1^{\bullet, 0})$ to get the bottom line of the $E_2$ page. Using the fact that $H^0(G^p,\G_m)=k^*\oplus X(G^p)=k^*$ is the set of constant functions on $G^p$ for all $p\geq 1$, we see that the pullback along any projection map $G^{p+1}\to G^{p}$ is the identity map $k^*\to k^*$. Thus, for all $p\geq 1$, $d^{p,0}$ is the alternating sum of $p+2$ copies of $\id: k^*\to k^*$, so is given by the identity map for $p$ odd and the zero map for $p$ even. For example, $d^{1,0}: H^0(G, \G_m)\to H^0(G\times G,\G_m)$ is given by $p_2^*-m^*+p_1^*$, which is the identity map $k^*\to k^*$, since each term in the sum is the identity map.

    Thus $H^p(E_1^{\bullet, 0})$ is the cohomology of
    \[\begin{tikzcd}
        k^* \arrow{r}{0} & k^* \arrow{r}{\id} & k^* \arrow{r}{0} & \cdots,
    \end{tikzcd}\]
    so $H^0(E_1^{\bullet, 0})= k^*$, and $H^p(E_1^{\bullet, 0})=0$ for $p\geq1$. 

    Next, we compute the cohomology of the middle line, $H^p(E_1^{\bullet, 1})$. We must understand the three pullback maps 
    \[p_1^*, p_2^*, m^*: \Pic(G)=X(B)\to \Pic(G\times G)=X(B\times B).\]
    Let $\chi\in X(B)$ correspond to $\chi_*\widetilde{G}\in \Pic(G)$. Write $\rho:\widetilde{G}\to G$ and $\rho':\chi_*\widetilde{G}\to G$. Then in $\Pic(G\times G)$,
    \[p_1^*(\chi_*\widetilde{G})=\chi_*\widetilde{G}\times_{G, p_1} (G\times G)\] 
    \[ p_2^*(\chi_*\widetilde{G})=\chi_*\widetilde{G}\times_{G, p_2} (G\times G), \]
    \[ m^*(\chi_*\widetilde{G})=\chi_*\widetilde{G}\times_{G, m} (G\times G)\]
    Each of the above is a $\G_m$-torsor over $G\times G$, which we can write $\chi_*\widetilde{G}\times G\to G\times G$ with the following structure maps,
    \[p_1^*(\chi_*\widetilde{G})=\chi_*\widetilde{G}\times G\to G\times G, \ (x,g)\mapsto (\rho'(x),g),\]
    \[p_2^*(\chi_*\widetilde{G})=\chi_*\widetilde{G}\times G\to G\times G, \ (x,g)\mapsto (g,\rho'(x)),\]
    \[m^*(\chi_*\widetilde{G})=\chi_*\widetilde{G}\times G\to G\times G, \ (x,g)\mapsto (g,g^{-1}\rho'(x)).\]

    Let $p_1^*\chi, p_2^*\chi, m^*\chi\in X(B\times B)$ correspond to $p_1^*(\chi_*\widetilde{G}), p_2^*(\chi_*\widetilde{G}), m^*(\chi_*\widetilde{G})$, respectively.

    These maps are given as follows,
    \[p_1^*\chi=\chi\circ p_1:  (b_1, b_2)\mapsto \chi(b_1),\]
    \[p_2^*\chi=\chi\circ p_2:  (b_1, b_2)\mapsto \chi(b_2),\]
    \[m^*\chi=\chi\circ m:  (b_1, b_2)\mapsto \chi(b_1b_2),\]
    since they pushout the $B\times B$-torsor $\widetilde{G}\times \widetilde{G}\to G\times G$ to the $\G_m$-torsors $p_1^*(\chi_*\widetilde{G}), p_2^*(\chi_*\widetilde{G}),m^*(\chi_*\widetilde{G}),$ respectively.
        
    Then $d^{1,1}: X(B)\to X(B\times B)$ is given by 
    \[d^{1,1}:\chi\mapsto p_2^*\chi-m^*\chi+p_1^*\chi:(b_1,b_2)\mapsto \chi(b_2)-\chi(b_1b_2)+\chi(b_1)=0,\]
    since $\chi$ is a homomorphism. Thus $H^p(E_1^{\bullet, 1})$ is given by the cohomology of the sequence 
    \[\begin{tikzcd} 0 \arrow{r} & X(B) \arrow{r}{0} & X(B\times B) \end{tikzcd}\]

    This gives $H^0(E_1^{\bullet, 1})=0$, $H^1(E_1^{\bullet, 1})=X(B)$.

    We can now write out the $E_2$ page, 
    \begin{center}
\begin{tikzpicture}[scale=1.5, every node/.style={font=\small}]
  % Axes
  \draw[->] (-0.5, 0) -- (3.5, 0) node[anchor=west] {$p$};
  \draw[->] (0, -0.5) -- (0, 2.5) node[anchor=south] {$q$};

  % Nodes
  \node at (0.25,0.25) {$k^*$};
  \node at (1.25,0.25) {$0$};
  \node at (2.25, 0.25) {0};
  \node at (3.25, 0.25) {0};
  \node at (0.25,1.25) {$0$};
  \node at (1.25,1.25) {$X(B)$};
  \node at (0.25, 2.25) {0};
  \node at (1.25, 2.25) {$\cdots$};

  % Arrow (differential)
  \draw[->] (0.5,1.1) -- (2.1,0.3);
  \draw[->] (1.5,1.1) -- (3.1, 0.3);
  
\end{tikzpicture}
\end{center}

    from which we can read off $H^i(BG, \G_m)$ for $i=0,1,2$. This concludes the proof of the first statement.

    We now show that the map 
    \[X(B)\to H^2(BG,\G_m), \ \chi\mapsto \chi_*(B\widetilde{G})\] 
    is an isomorphism. Let $\ms H\to BG$ be a $\G_m$-gerbe. Take the fiber product with $B\widetilde{G}\to BG$ to get $\ms H\times_{BG} B\widetilde{G}\to B\widetilde{G}$, a $\G_m$-gerbe over $B\widetilde{G}$. By the first part of the proof, $H^2(B\widetilde{G},\G_m)=0$ since the fundamental group of $\widetilde{G}$ is 0. Thus, our $\G_m$-gerbe $\ms H\times_{BG} B\widetilde{G}\to B\widetilde{G}$ is trivial, and thus has a section. Composing the section with the projection to $\ms H$ gives a map $B\widetilde{G}\to \ms H$. The map on stabilizers over $BG$ is a map $\chi: B\to \G_m$ such that $\chi_*(B\widetilde{G})=\ms H$. Thus the map is surjective. 

    Now let $\chi_1:B\to \G_m$ and $\chi_2:B\to \G_m$ be two characters such that $\ms H_1=\chi_{1*}(B\widetilde{G})$ is isomorphic to $\ms H_2=\chi_{2*}(B\widetilde{G})$. That is, there is an  isomorphism of $\G_m$-gerbes $\sigma:\ms H_1\to \ms H_2$ on $BG$ and an isomorphism 
    \[\tau:(\ms H_1\to BG)\cong (\ms H_1\xrightarrow{\sigma}\ms H_2 \to BG).\]
    From $\chi_1$ and $\chi_2$, we get morphisms $\phi_1: B\widetilde{G}\to \ms H_1$ and $\phi_2: B\widetilde{G}\to \ms H_2$ whose morphisms on stabilizers are $\chi_1$ and $\chi_2$ respectively. This defines $\phi_1\times \phi_2: B\widetilde{G}\to \ms H_1\times_{BG} \ms H_2$.

    We have the following equivalence of categories, 
    \[\ms H_1\times_{BG} \ms H_2\to \ms H_1\times B\G_m,\]
    where an equivalence is given by sending 
    \[\left(h_1:S\to \ms H_1, h_2:S\to \ms H_2,\alpha:(S\xrightarrow{h_1} \ms H_1\to BG)\cong (S\xrightarrow{h_2}\ms H_2\to BG)\right)\]
    to $\left(h_1, \uIsom_{\ms H_2}(\sigma\circ h_1, h_2)\right)$, where $\uIsom_{\ms H_2}(\sigma\circ h_1,h_2):(\Sch/S)^\op \to \Set$ is a $\G_m$-torsor defined as follows:
    \[(f:S'\to S)\mapsto \left\{\eta:\left(S'\xrightarrow{f}S\xrightarrow{h_1}\ms H_1 \xrightarrow{\sigma}\ms H_2\right)\cong \left(S'\xrightarrow{f}S\xrightarrow{h_2}\ms H_2\right) \middle|  (\ms H_2\to BG)\circ \eta=f^*(\alpha\circ h_1^*\tau^{-1})  \right\},\]
    where $\alpha\circ h_1^*\tau^{-1}$ is the composition of the maps
    \[\left(S\xrightarrow{h_1} \ms H_1 \xrightarrow{\sigma} \ms H_2 \to BG\right) \xrightarrow{h_1^*\tau^{-1}}\left(S\xrightarrow{h_1} \ms H_1\to BG\right)\xrightarrow{\alpha}\left(S\xrightarrow{h_2}\ms H_2 \to BG\right).\]
    The $\G_m$-torsor $\uIsom_{\ms H_2}(\sigma\circ h_1,h_2)$ can be thought of as the torsor of isomorphisms between $h_1$ and $h_2$ whose image in $BG$ is given by $\alpha$.

    A morphism in $\ms H_1\times_{BG}\ms H_2$ is a pair 
    \[(\beta_1,\beta_2):(h_1,h_2,\alpha)\to (h_1',h_2',\alpha'),\]
    where $\beta_1:h_1\to h_1'$ and $\beta_2:h_2\to h_2'$ are isomorphisms such that the diagram
    \begin{equation} \label{eq: isom diagram}
    \begin{tikzcd}
        \overline{h_1} \arrow{r}{\alpha} \arrow{d}[swap]{\overline{\beta_1}} & \overline{h_2} \arrow{d}{\overline{\beta_2}} \\ \overline{h_1'} \arrow{r}{\alpha'} & \overline{h_2'}
    \end{tikzcd}
    \end{equation}
    commutes, where the bar denotes composition with the appropriate map to $BG$. 

    A morphism $(\beta_1,\beta_2)$ in $\ms H_1\times_{BG} \ms H_2$ is sent to $(\beta_1,\gamma)$ in $\ms H_1\times B\G_m$, where \[\gamma: \uIsom_{\ms H_2}(\sigma\circ h_1,h_2)
    \to \uIsom_{\ms H_2}(\sigma\circ h_1',h_2')\]
    is given by 
    \[(\eta: f^*(\sigma\circ h_1)\cong f^*h_2)\mapsto \eta':=f^*\beta_2\circ \eta\circ f^*(\sigma\circ \beta_1)^{-1}.\]
    Note that 
    \begin{align*}
        \overline{\eta'}&= f^*\overline{\beta_2}\circ \overline{\eta}\circ f^*(\overline{\sigma\circ \beta})^{-1} \\ 
        &=f^*\overline{\beta_2}\circ f^*(\alpha\circ h_1^*\tau^{-1})\circ f^*(\overline{\sigma\circ \beta})^{-1} \\
        &=f^*\overline{\beta_2}\circ f^*\alpha\circ f^*h_1^*\tau^{-1}\circ f^*(\overline{\sigma\circ \beta})^{-1} \\
        &=f^*\overline{\beta_2}\circ f^*\alpha\circ f^* \overline{\beta_1}^{-1} \circ  f^*h_1'^*\tau^{-1} \\
        &=f^*\alpha' \circ f^*h_1'^*\tau^{-1}
    \end{align*} due to the commutativity of the diagram \ref{eq: isom diagram}, and functoriality of $\tau$. 

    We show that the map is fully faithful and essentially surjective. let $(h_1, h_2, \alpha), (h_1',h_2',\alpha')\in \ms H_1\times_{BG}\ms H_2$. Given a morphism $\beta:h_1\to h_1'$ and $\gamma:\uIsom_{\ms H_2}(\sigma\circ h_1,h_2)\to \uIsom_{\ms H_2}(\sigma\circ h_1',h_2')$, we may set $\beta_1=\beta$, and construct $\beta_2$ locally as follows. Choose a cover $f:S'\to S$ where $\uIsom_{\ms H_2}(\sigma\circ h_1,h_2)$ has a section $\eta:f^*(\sigma\circ h_1)\to f^*h_2$. We may define $\beta_2$ locally by $\gamma(\eta)\circ f^*(\sigma\circ \beta) \circ \eta^{-1}$. Then by descent on $\ms H_2$, there exists a unique morphism $\beta_2: h_2\to h_2'$. Thus the map is fully faithful.

    For essential surjectivity, let $(h:S\to \ms H_1, T: S\to B\G_m)\in\ms H_1\times B\G_m$. Let $f:S'\to S$ be a cover of $S$ where $T$ becomes the trivial torsor, $T|_{S'}=S'\times \G_m$. Then $\left(f^*h, f^*(\sigma\circ h), f^*h^*\tau:f^*\overline{h}\to f^*\overline{(\sigma\circ h})\right)$ maps to $(f^*h, S'\times \G_m)$ since $\uIsom_{\ms H_2}(f^*(\sigma\circ h), f^*(\sigma\circ h))$ has a section, namely the identity map. 

    We may now use our equivalence of categories and the map $\phi_1\times \phi_2:B\widetilde{G}\to \ms H_1\times_{BG} \ms H_2$ to get a map $B\widetilde{G} \to \ms H_1\times B\G_m$. By the calculation in the first part of the proof, $H^1(B\widetilde{G}, \G_m)=0$, so the map is given by some $\phi: B\widetilde{G} \to \ms H_1$ and the trivial torsor $B\widetilde{G}\times \G_m$. Since the image consists of the trivial torsor, this tells us that $\uIsom_{\ms H_2}(\sigma\circ \phi_1,\phi_2)$ has a section over $B\widetilde{G}$, say $\eta: (B\widetilde{G}\xrightarrow{\phi_1}\ms H_1 \xrightarrow{\sigma} \ms H_2)\to (B\widetilde{G}\xrightarrow{\phi_2}\ms H_2).$ 

    Recall that $\chi_1,\chi_2:B\to \G_m$ describe the map on stabilizers over $BG$ of $\phi_1: B\widetilde{G}\to \ms H_1$ and $\phi_2:B\widetilde{G}\to \ms H_2$, respectively. Let $b\in B$ and $T:S\to B\widetilde{G}$, and view $b: T\cong T$ as an isomorphism in $B\widetilde{G}$ over $BG$. Then $\chi_1(b): \phi_{1}(T)\to \phi_{1}(T)$ and $\chi_2(b): \phi_{2}(T)\to \phi_{2}(T)$ are isomorphisms in $\ms H_1$ and $\ms H_2$ over $BG$, respectively. The following diagram commutes,
    \[\begin{tikzcd}
        \phi_{1}(T) \arrow{r}{\sigma} \arrow{d}[swap]{\chi_1(b)} & (\sigma\circ \phi_1)(T) \arrow{r}{\eta} & \phi_{2}(T) \arrow{d}{\chi_2(b)} \\ \phi_{1}(T)\arrow{r}{\sigma} & (\sigma\circ \phi_1)(T) \arrow{r}{\eta} & \phi_{2}(T),
    \end{tikzcd}\]
    with $\eta$ and  $\sigma$ isomorphisms. Then, since $\G_m$ is abelian, we conclude that $\chi_1(b)=\chi_2(b)$, and thus our map $X(B)\to H^2(BG,\G_m)$ is injective.
\end{proof}

\begin{remark}
    One could also formally use the spectral sequence to determine the isomorphism $X(B)\to H^2(BG,\G_m)$, but we have shown directly that the given map is an isomorphism. 
\end{remark}

\section{Brauer-Severi stacks over $BG$} \label{sec: Br-Sev stacks}
\begin{definition}
    Let $S$ be an algebraic stack. A \textit{Brauer-Severi stack of relative dimension $n$ over $S$} is a smooth, proper, morphism $\pi: \ms X\to S$ of relative dimension $n$ such that there exists an \'etale cover $S'\to S$ where $\ms X_{S'}\cong \PP^n_{S'}$. 
\end{definition}

\begin{remark}
    A smooth, proper morphism $\pi:\ms X\to S$ such that for every geometric point $s:\Spec k\to S$, $\ms X_s\cong \PP^n_{k}$, is a Brauer-Severi stack.
\end{remark}

The following two propositions provide equivalent ways to view Brauer-Severi stacks in general, or Brauer-Severi stacks that are parameterized by a vector bundle. 
\begin{proposition} \label{prop: br-sev eq pgl}
    There is an equivalence of categories between Brauer-Severi stacks of relative dimension $n$ over $S$ and $\PGL_{n+1}$-torsors over $S$. 
\end{proposition}
\begin{proof}
    Given a Brauer-Severi stack $\pi: \ms X\to S$, we define a $\PGL_{n+1}$-torsor $I_{\ms X}$ on $S$ to be 
\[I_{\ms X}(T\to S)=\{ \text{isomorphisms } \sigma: \ms X\times_S T\cong \mathbb{P}^n_T\},\]
where $\alpha\in \PGL_{n+1}(T)=\Aut(\PP^n_T)$ acts by composition with $\sigma$. The equivalence is given by sending $\pi:\ms X\to S$ to $I_{\ms X}$, and an isomorphism of Brauer-Severi stacks $\phi:\ms X\to \ms X'$ is sent to $\phi_*: I_{\ms X}\to I_{\ms X'}$, given by 
\[(\sigma:\ms X\times_S T\cong \PP^n_T)\mapsto (\sigma\circ(\phi^{-1}\times \id):\ms X'\times_S T\cong \ms X\times_ST \cong \PP^n_T). \qedhere\]
\end{proof}

\begin{example} \label{ex: M_0=BPGL_2}
    For $n=1$, this is the statement that the moduli stack of genus 0 curves $\ms M_0$ is isomorphic to $B\PGL_2$. 
\end{example}
\begin{proposition} \label{prop: br sev eq gl_n}
    There is an equivalence of categories between rank $n+1$-vector bundles on $S$, $\GL_{n+1}$-torsors over $S$, and Brauer-Severi stacks of relative dimension $n$ together with a degree 1 line bundle over $S$. 
\end{proposition}

\begin{proof}
    Given a rank $n+1$-vector bundle $V$ on $S$, we define a $\GL_{n+1}$-torsor $I_V$ on $S$ to be 
    \[I_V(T\to S)=\{\text{isomorphisms }\sigma: V|_T\cong \ms O_T^{\oplus n+1}\},\]
    where $\alpha \in \GL_{n+1}(T)=\Aut(\ms O_T^{\oplus n+1})$ acts by composition with $\sigma$.
    The equivalences are given by sending $V$ to $I_V$, and by sending $V$ to $(\PP V, \ms O_{\PP V}(1))$.     
\end{proof}

\begin{definition}
    The image of the class of $\pi:\ms X\to S$ in $H^1(S, \PGL_{n+1})$ under the differential coming from the short exact sequence
    \[1\to \G_m \to \GL_{n+1} \to \PGL_{n+1}\to 1\] 
    is the class $\delta([\ms X])\in H^2(S,\G_m)$, which we call the \textit{Brauer class of the Brauer-Severi stack $\ms X\to S$}.
\end{definition}

\begin{proposition}
    The Brauer class of the Brauer-Severi stack $\ms X\to S$ is given by the class $[\ms G]=\delta([\ms X])\in H^2(S,\G_m)$ of the following $\G_m$-gerbe over $S$,
    \[\ms G(T\to S)=\{\text{degree 1 line bundles on } \ms X\times_S T\}.\]
    Morphisms in $\ms G$ are isomorphisms of line bundles. 
    In particular, $\ms G$ is a trivial gerbe if and only if $\ms X\cong \PP V$ for some vector bundle $V$ on $S$.
\end{proposition}

\begin{proof}
    The class $\delta([\ms X])$ is the $\G_m$-gerbe 
    \[\delta(I_{\ms X})(T\to S)=\{(P, \epsilon:P^{\wedge \PGL_{n+1}}\cong I_{\ms X}|_T\},\]
    where $P$ is a $\GL_{n+1}|_T$-torsor on $T$, and $\epsilon$ is an isomorphism of the pushout of $P$ to $\PGL_{n+1}$ with $I_{\ms X}|_T$. Morphisms in $\delta(I_{\ms X})$ are isomorphisms of $\GL_{n+1}$-torsors that are compatible with the isomorphisms with $I_{\ms X}$. By Proposition \ref{prop: br sev eq gl_n}, we see that the data of a $\GL_{n+1}|_T$-torsor together with an isomorphism of its $\PGL_{n+1}$-pushout to $I_{\ms X}$ is equivalent to the data of a rank $n+1$-vector bundle $V$ on $T$ together with an isomorphism $\PP V\cong \ms X \times_S T$, which is equivalent to the data of a Brauer-Severi variety $Y\to T$ together with a degree 1 line bundle on $Y$ and an isomorphism $Y\cong \ms X\times_S T$, which is equivalent to the data of a degree 1 line bundle on $\ms X\times_S T$.

    Moreover, $\ms G$ is trivial if and only if $[\ms G]=\delta ([\ms X])=0$ if and only if there is a $\GL_{n+1}$ torsor over $S$ whose pushout to $\PGL_{n+1}$ is $I_{\ms X}=[\ms X]$ if and only if there is a vector bundle $V$ on $S$ such that $\ms X\cong \PP V$. 
\end{proof}
 
\subsection{} \label{sec: PV} From this point forward, we will only consider Brauer-Severi stacks over $BG$ of the form $[\PP V/G]\to BG$, where $G$ is as in Section \ref{section: Br(BG)} and $V$ is a representation of the universal cover $\widetilde{G}$. We will use Theorem \ref{thm: Br(BG) main} to describe their Brauer classes. Let $V$ be a representation of $\widetilde{G}$, acting (on the left) via $\rho: \widetilde{G}\to \GL(V)$ such that $B\subseteq \widetilde{G}$ acts by a character $\chi:B\to \G_m.$ Since $B$ acts by a character, we get an induced action of $G$ on $(V\setminus\{0\} )/\G_m$. It will be useful to understand $V$ and $(V\setminus\{0\} )/\G_m$ functorially, so we introduce $\mathbb{A}_V$ and $\PP V$ to represent these.  

Let $\mathbb{A}_V$ represent the functor 
\[\mathbb{A}_V: (\Sch/\Spec k)^\op \to \Set \]
\[(T\to \Spec k)\mapsto \Hom_k(V\otimes_k \ms O_T, \ms O_T)\]

We have $\bbA_V(k)=\Hom_k(V,k)=V^\vee$, and $\bbA_V=\Spec(\Sym^\bullet_kV)$.
Indeed, 
\begin{align*}
    (\Spec(\Sym^\bullet_kV))(T)= & \ \Hom(\Sym^\bullet_kV, \Gamma(T,\ms O_T)) \\
    = & \ \Hom(\Sym^\bullet_k V\otimes_k \ms O_T, \ms O_T) \\ = & \ \Hom(V\otimes_k \ms O_T, \ms O_T) \\ = & \ \bbA_V(T).
\end{align*} 
By $0$ in $\bbA_V(k)$, we mean the zero map $V\to k$. More generally, by 0 in $\bbA_V(T)$ we mean the maps $V\otimes_k \ms O_T\to \ms O_T$ which are not surjective. 

On $k$-points, the action of $\widetilde{G}$ on $V^\vee$ induced by the action of $\widetilde{G}$ on $V$ is given by composition with the action map, and makes the action on $V^\vee$ into a right action. Explicitly, for $g\in \widetilde{G}$ and $f:V\to k$ in $V^\vee$, $g\cdot f: V\to k$, is given by $v\mapsto f(g\cdot v)$. For $b\in B$, $b\cdot f:V\to k$ is given by $v\mapsto f(b\cdot v)=f(\chi(b)v)=\chi(b)f(b)$, so $B$ also acts by $\chi$ on $V^\vee$. 

\begin{remark}
    If one prefers only to work with left actions, then we can make the action of $\widetilde{G}$ on $V^\vee$ into a left action by composing with the inverse action map instead. Explicitly, $g\cdot f:V\to k$ is given by $v\mapsto f(g^{-1}\cdot v)$. In this case, $B$ acts by $\chi^{-1}$. 
\end{remark}

The action of $\widetilde{G}$ makes sence functorially, on $\bbA_V$, as well. Explicitly, $g\cdot (f:V\otimes_k\ms O_T\to \ms O_T)$ is the composition of $f$ with 
\[(v\mapsto g\cdot v)\otimes \id: V\otimes_k \ms O_T\to V\otimes_k \ms O_T.\]

Let $\PP V$ represent the functor 
\[\PP V: (\Sch/ \Spec k)^{\op} \to \Set\]
\[ (T\to \Spec k )\mapsto \{ V\otimes_k \ms O_T \twoheadrightarrow\ms L\}/\sim \]
where two surjections $(V\otimes_k \ms O_T \twoheadrightarrow\ms L_1)$ and $(V\otimes_k \ms O_T \twoheadrightarrow\ms L_2)$ are equivalent if there is an isomorphism $s:\ms L_1\to \ms L_2$ making the diagram
\[\begin{tikzcd}
    V\otimes_k \ms O_T \arrow{r} \arrow{dr} & \ms L_1 \arrow{d}{s} \\ \ & \ms L_2
\end{tikzcd}\]

commute. 

Note that $\PP V$ is the sheafification of the functor
\[ (\PP V)^{\pre}: (\Sch/\Spec k)^\op \to \Set \]
\[ (T\to \Spec k)\mapsto \{ V\otimes_k\ms O_T\twoheadrightarrow\ms O_T\}/\sim\]
with the same equivalence relation. That is, $\PP V\cong (\mathbb{A}_V\setminus \{0\})/\G_m$.

The action of $\widetilde{G}$ on $\PP V$ is similarly given by composition with the action on $V$. Since $B$ acts by a character, $B$ acts trivially on the quotient, and thus the action of $\widetilde{G}$ descends to an action of $G$ on $\PP V$. 

Consider the stack quotient $[\mathbb{P}V/G]\to BG$, which is a Brauer-Severi stack over $BG$, since pulling back along $\Spec k\to BG$ gives $\PP V\to \Spec k$. The Brauer class of $[\PP V/G]$ is given by the $\G_m$-gerbe $\ms G$ of degree 1 line bundles on $[\PP V/G]$ over $BG$,
\begin{equation} \label{eq: ms G}
    \ms G(S\to BG)=\{\text{degree 1 line bundles on }S\times_{BG}[\mathbb{P}V/G]=:\mathbb{P}_S\}.
\end{equation}

\begin{lemma} \label{lemma: P_S}
    The fiber product $\mathbb{P}_S:=S\times_{BG}[\mathbb{P}V/G]$ is given by $T\times^G \PP V:=(T\times \PP V)/G$, where $T\to S$ is the $G$-torsor corresponding to $S\to BG$ and  $G$ acts on $T\times \PP V$ via $g\cdot(t,v)=(g\cdot t,g\cdot v)$. 
\end{lemma}

\begin{proof}
    The objects of the fiber product $\PP_S$ can be described as tuples
    \[(S'\to S, \ T'\to S', \  \varphi:T'\to \PP V, \ \sigma: T\times_S S'\cong T')\]
    where $S'$ is an $S$-scheme, $T'\to S'$ is a $G$-torsor, $\varphi$ is a $G$-equivariant morphism, and $\sigma$ an isomorphism of $G$-torsors over $S'$. 

    Morphisms between two tuples 
    $(S'\to S, \ T'\to S', \ \varphi, \ \sigma)\to (S''\to S, \ T''\to S'',\  \varphi'', \ \sigma'')$ are given by pairs $(s,t)$, where $s: S'\to S''$ is a morphism of $S$-schemes, and $t:T'\to T''$ is a $G$-equivariant morphism, such that the following diagrams commute,
    \[\begin{tikzcd}
        T'\arrow{rr}{t} \arrow{dr}[swap]{\varphi} && T''\arrow{dl}{\varphi''} \\ \  & \PP V, & \ 
    \end{tikzcd} \ \ \ \ 
    \begin{tikzcd}
        T\times_S S' \arrow{r}{\sigma} \arrow{d}[swap]{\id\times s} & T' \arrow{d}{t} \\ T\times_S S'' \arrow{r}{\sigma''} & T''.
    \end{tikzcd}\]
    
    Note that due to the second compatibility condition, $t$ is completely determined by $t=\sigma'' \circ (\id\times s)\circ \sigma^{-1}$. 

    We can see that this data is equivalent to the stack with objects given by pairs 
    $$(S'\to S, \phi: T\times_S S'\to \PP V)$$
    where $S'$ is an $S$-scheme and $\phi$ is a $G$-equivariant morphism, and morphisms between two tuples 
    $(S'\to S, \ \phi)\to (S''\to S, \phi'')$ are morphisms of $S$-schemes $s: S'\to S''$ such that the diagram commutes,
    \[\begin{tikzcd}
        T\times_S S'\arrow{rr}{\id\times s} \arrow{dr}[swap]{\phi} && T\times_S S''\arrow{dl}{\phi''} \\ & \PP V. &
    \end{tikzcd}\]

    Indeed, an equivalence is given by sending $(S'\to S, T'\to S', \varphi, \sigma)$ to $(S'\to S, \varphi\circ\sigma)$, and on morphisms, $(s,t)\mapsto s$.

    This is fully faithful, since $t$ is determined as remarked above, and essentially surjective since $(S'\to S, T\times_S S'\to S', \phi, \id)\mapsto (S'\to S, \phi)$. 

    Now we define an equivalence between $T\times^G \PP V$ and $\PP_S$. 
    
    First, consider the $G$-torsor $T\times \PP V\to T\times^G \PP V$. With the action of $G$ on $T\times \PP V$ as described above, we get both projection maps, from $T\times \PP V$ to $T$ and $\PP V$, to be $G$-equivariant morphisms. Taking the quotient of $T\times \PP V\to T$ by $G$ gives the following cartesian square,
    \[\begin{tikzcd}
        T\times \PP V \arrow{r} \arrow{d} & T \arrow{d} \\ T\times^G \PP V \arrow{r} & S.
    \end{tikzcd}\]

    Given a morphism $S'\to T\times^G \PP V$, compose with the map to $S$ to get $s:S'\to S$. The fiber product $S' \times_{T\times^G \PP V} (T\times \PP V)$ is given by $T\times_S S'$ due to the commutativity of the above square, and we get a $G$-equivariant map, $T\times_S S'\to T\times \PP V \to \PP V$. Indeed, since $T\times_S S'\to T$ and $T\times \PP V\to T$ are both $G$-equivariant, $T\times_S S'\to T\times \PP V$ is $G$-equivariant, and thus the composition to $\phi: T\times_S S'\to \PP V$ is as well.  

    Conversely, given $s:S'\to S$ and $\phi:T\times_S S'\to \PP V$, take the quotient of $\mathrm{pr}_1 \times \phi: T\times_S S'\to T\times \PP V$ by the action of $G$ to get $S'\to T\times^G \PP V$. These are clearly inverses of each other. 
\end{proof}

\begin{remark}
    One can think of the category of $G$-equivariant morphisms $T\to \PP V$ as pairs $(t,v)$ of input and output, which are well-defined up to the equivalence $(g\cdot t, g\cdot v)\sim (t,v)$, coinciding with the local description of $T\times^G \PP V$. 
\end{remark}

\begin{remark}
    The following commutative diagram can also be used as a justification of Lemma \ref{lemma: P_S},
    \[\begin{tikzcd}
        T\times \PP V \arrow{rr} \arrow{dr}\arrow{dd} && \PP V \arrow{dr}\arrow{dd} & \ \\ & T \arrow{rr}\arrow{dd} && \Spec k \arrow{dd} \\ \PP_S \arrow{rr}\arrow{dr} && {[\PP V/G]} \arrow{dr} & \\ \ & S \arrow{rr} && BG.
    \end{tikzcd}\]
\end{remark} 
\begin{lemma}
    $\bbA_V\setminus\{0\}\to \PP V$ is a $\G_m$-torsor, corresponding to the line bundle $\ms O_{\PP V}(1)$ on $\PP V$, that is, 
    $$\bbA_V\setminus\{0\} =\Tot(\ms O_{\PP V}(1))\setminus\{0\}:=\uSpec_{\PP V}(\Sym^\bullet_k\ms O_{\PP V}(-1))\setminus \{0\}.$$
\end{lemma}
\begin{proof} 
    The map $\bbA_V\setminus\{0\}\to \PP V$ is given by 
    \[(V\otimes_k\ms O_T \twoheadrightarrow\ms O_T)\mapsto [V\otimes_k\ms O_T \twoheadrightarrow\ms O_T]\]
    and objects of $\bbA_V\setminus\{0\}$ over a point $f:T\to \PP V$, corresponding to $\varphi: V\otimes_k \ms O_T\twoheadrightarrow\ms L$ pulled back from the universal quotient $V\otimes_k\ms O_{\PP V}\twoheadrightarrow\ms O_{\PP V}(1)$, are given by 
    \[(\bbA_V\setminus\{0\})(T\to \PP V)=\left\{ \phi: V\otimes_k \ms O_T\twoheadrightarrow\ms O_T| \exists s: 
    \begin{tikzcd}
        V\otimes_k \ms O_T \arrow{r}{\phi} \arrow{dr}[swap]{\varphi} & \ms O_T \arrow{d}{s} \\ \ & \ms L
    \end{tikzcd}
    \text{ commutes}\right\}.\]
    This data is equivalent to
    the set of isomorphisms $\{s: \ms O_T\cong \ms L\}$.

    Objects of $\Tot(\ms O_{\PP V}(1))\setminus \{0\}$ over a point $f:T\to \PP V$ are given by 
    \[\left(\uSpec_{\PP V}(\Sym^\bullet_k\ms O_{\PP V}(-1))\setminus\{0\}\right)(T\to \PP V)=\{\Sym^\bullet \ms L^\vee \twoheadrightarrow \ms O_T\}\cong \{\ms L^\vee \twoheadrightarrow\ms O_T\}\cong \{\ms O_T\cong \ms L\}.\]
    Thus $\bbA_V\setminus\{0\} =\Tot(\ms O_{\PP V}(1))\setminus\{0\}$.    
\end{proof}

\begin{theorem} \label{thm: chi(BG)=ms G}
    Let $G, \widetilde{G}, B, V, \chi$ as in Section \ref{sec: PV} and $\ms G$ as in Equation \ref{eq: ms G}. Then $\chi^{-1}_*(B\widetilde{G})=\ms G$, and thus by Theorem \ref{thm: Br(BG) main}, the class of $\ms G$ in $H^2(BG, \G_m)=X(B)$ is given by $\chi^{-1}$.
\end{theorem}

\begin{proof}
    We define a morphism of stacks $B\widetilde{G}\to \ms G$ over $BG$ and show that the morphism on stabilizers is given by $\chi^{-1}$. Given $S\to B\widetilde{G}$ corresponding to a $\widetilde{G}$-torsor $\widetilde{T}\to S$, we must assign this to a degree 1 line bundle on $\PP_S=T\times^G\PP V$, where $T$ is the image of $\widetilde{T}$ in $BG$. Let its image in $\ms G$ be the line bundle corresponding to the $\G_m$-torsor 
    \[\widetilde{T}\times^{\widetilde{G}} \bbA_V\setminus \{0\} \to T\times^G\PP V=\widetilde{T}\times^{\widetilde{G}}\PP V.\] 
    Note that this map appears in the cartesian diagram

    \[\begin{tikzcd}
        \widetilde{T}\times^{\widetilde{G}} (\bbA_V\setminus \{0\}) \arrow{rr} \arrow{d} && {[\bbA_V\setminus\{0\} /\widetilde{G}]} \arrow{d} &&   \\ T\times^G \PP V =\widetilde{T}\times^{\widetilde{G}}\PP V \arrow{rr} \arrow{d} && {[\PP V/\widetilde{G}]} \arrow{rr} \arrow{d} && {[\PP V/G]} \arrow{d} \\ S \arrow{rr} && B\widetilde{G} \arrow{rr} && BG.
    \end{tikzcd}\]

    A morphism $b: \widetilde{T} \to \widetilde{T}'$ of $\widetilde{G}$-torsors whose pushouts to $G$ agree is sent to the morphism 
    \[\widetilde{T}\times^{\widetilde{G}} (\bbA_V\setminus\{0\}) \to \widetilde{T'}\times^{\widetilde{G}}(\bbA_V\setminus\{0\}),\]
    which we describe on the cover by
    \[\widetilde{T}\times (\bbA_V\setminus\{0\}) \to \widetilde{T}'\times (\bbA_V\setminus\{0\}),\ \ (t,v)\mapsto (b\cdot t,v).\]
    Upon taking the quotient by the $\widetilde{G}$-action, the map is $[(t,v)]\mapsto [(b\cdot t,v)]=[(t,b^{-1}\cdot v)]=[(t,\chi(b)^{-1}v)]$.

    Thus, the morphism on automorphism groups $B\to \G_m$ for $B\widetilde{G}\to \ms G$ over $BG$ is given by $\chi^{-1}$, so $\chi^{-1}_*(B\widetilde{G})=\ms G$.
\end{proof}

More generally, let $V_1,\ldots, V_r$ be representations of $\widetilde{G}$, acting via $\rho_i: \widetilde{G}\to \GL(V_i)$ such that $B$ acts by characters $\chi_i$, for $i=1,\ldots, r$. Again the action descends to $\PP V_i$. Consider 
\[[(\PP V_1\times \cdots \times \PP V_r)/G]\to BG, \]
and the $\G_m$-gerbe over $BG$,
\[\ms G^{a_1,\ldots, a_r}(S\to BG)=\Biggl\{ {\ms L \text{ on } S\times_{BG}[\mathbb{P}V_1\times \cdots \PP V_r/G]=\mathbb{P}_S \text{ such that } \atop\forall s:\Spec \overline{k} \to S, \ \ms L|_{\PP_s}\cong \ms O(a_1,\ldots, a_r)}\Biggl\}.\]

Here $\ms O(a_1,\ldots, ,a_r)$ is the line bundle on $\PP V_1\times \cdots \times \PP V_r$ given by the tensor product of the pullbacks of $\ms O(a_i)$ on $\PP V_i$. To simplify notation, we will write $\ms O(a_1,\ldots, a_r)$ on $\PP_S$ for such a line bundle. 

\begin{corollary} \label{cor: V_1,...V_r}
    Let $\psi=\sum_{i=1}^r a_i\chi_i^{-1}$. Then $\psi_*(B\widetilde{G})=\ms G^{a_1,\ldots, a_r}$, so the class of $\ms G^{a_1,\ldots, a_r}$ is given by $\psi\in \widehat{B}$.
\end{corollary} 

\begin{proof}
    Using the result from Theorem \ref{thm: chi(BG)=ms G}, we only need to show that $\ms G^{a_1,\ldots, a_r}$ is the product of $\G_m$-gerbes $\ms G^{a_1},\ldots, \ms G^{a_r}$, to be defined below, and each $\ms G^{a_i}$ is the product $a_i\ms G_i$, where $\ms G_i$ is the $\G_m$-gerbe 
    \[\ms G_i(S\to BG)=\{ \ms O(1) \text{ on } S\times_{BG}[\mathbb{P}V_i/G]\}.\]
    as in Theorem \ref{thm: chi(BG)=ms G}. 

    Let $\ms G^{a_i}$ be the following $\G_m$-gerbe over $BG$,
    \[\ms G^{a_i}(S\to BG)=\{ \ms O(a_i) \text{ on } S\times_{BG} [\PP V_i/G]\}.\]

    Note that 
    \[ [\PP V_1 \times \cdots \times \PP V_r/G]\cong [\PP V_1 /G]\times_{BG} \cdots \times_{BG} [\PP V_r/G],\]
    and let $p_i:[\PP V_1 \times \cdots \times \PP V_r/G]\to [\PP V_i/G]$ be the projection maps. To show that $\ms G^{a_1,\ldots, a_r}\cong \ms G^{a_1}\cdots \ms G^{a_r}$, we define a morphism of gerbes 
    \[\ms G^{a_1} \times \cdots \times \ms G^{a_r}\to \ms G^{a_1,\ldots, a_r}\]
    such that the map on stabilizers is the product map $\G_m\times \cdots \times \G_m\to \G_m$. The map is given by 
    \[(\ms L_1,\ldots, \ms L_r)\mapsto p_1^*\ms L_1\otimes \cdots \otimes p_r^*\ms L_r.\]
    Similarly, to show that $\ms G^{a_i}\cong a_i\ms G_i$, we define a morphism $\ms G_i\times \cdots \times \ms G_i \to \ms G^{a_i}$, with $a_i$ copies of $\ms G_i$, such that the map on stabilizers is the product map. The map is given by 
    \[(\ms L_1,\ldots, \ms L_{a_i})\mapsto \ms L_1\otimes \cdots \otimes \ms L_{a_i}. \qedhere\]    
\end{proof}

\begin{example} \label{example:PGL_2}
    Let $G=\PGL_2$, so $\widetilde{G}=\SL_2$, and $B=\mu_2$. Let $V$ be the standard representation of $G$, so $V=k^{\oplus 2}$, and $\SL_2$ acts via 
$\begin{pmatrix}
A & B\\
C & D
\end{pmatrix}\begin{pmatrix}
    x \\ y
\end{pmatrix}= \begin{pmatrix}
    Ax+By \\ Cx+Dy
\end{pmatrix}$. Then $\mu_2$ acts by the standard character, and the action descends to an action of $\PGL_2$ on $\PP V$. The quotient $[\PP V/ \PGL_2]\to B\PGL_2$ is the universal curve over the moduli of genus 0 curves (See Proposition \ref{prop: genus 0 simplifications}, (2)). Let $\ms G$ be the $\G_m$-gerbe of degree 1 line bundles on $[\PP V/\PGL_2]$ over $B\PGL_2$, 
$$\ms G(S\to B\PGL_2)=\{ \text{degree 1 line bundles on } S\times_{B\PGL_2}[\mathbb{P}V/\PGL_2]=:\mathbb{P}_S\},$$
which we recognize as $B\GL_2$.
 Theorem \ref{thm: chi(BG)=ms G} confirms that $\chi_*(B\SL_2)=\ms G\cong B\GL_2$, where $\chi:\mu_2\to \G_m$ is the standard character ($\chi=\chi^{-1}$ here). Note that $B\GL_2$ is indeed a nontrivial gerbe, due to the existence of genus 0 curves over non-algebraically closed fields without line bundles of degree 1.
\end{example}

\begin{example} \label{example:Sym^d}
    The next simplest case to understand is $G=\PGL_2$, as above, and $V=\Sym^d(k^{\oplus 2})=H^0(\PP k^{\oplus 2}, \ms O(d))\cong k^{\oplus d+1}$. We can think of $\PP V$ as the space parameterizing degree $d$ divisors on $\PP k^{\oplus 2}\cong\PP^1$, coming from the vanishing of homogeneous degree $d$ polynomials in two variables, with $d+1$ coefficients, giving a unique polynomial up to scalar multiple. The action of $\mu_2$ on $V$ is induced by the action on $k^{\oplus 2}$ and is given by $(-1)^d$.

    The action descends to an action of $\PGL_2$ on $\PP V$, and the quotient, $[\PP V/\
    \PGL_2]\to B\PGL_2$, is the moduli of degree $d$ divisors on genus 0 curves (See Proposition \ref{prop: genus 0 simplifications} (4)). Let $\ms G_d$ be the $\G_m$-gerbe of degree 1 line bundle on $[\PP V/\PGL_2]$ over $B\PGL_2$,
\begin{align*}
    \ms G_d(S\to B\PGL_2)&=  \{ \text{degree 1 line bundles on } S\times_{B\PGL_2}[\mathbb{P}\Sym^d(k^{\oplus 2})/\PGL_2]\} \\
    & \cong \{ \text{degree } d \text{ line bundles on } S\times_{B\PGL_2}[\mathbb{P}k^{\oplus 2}/\PGL_2]\} \\
    & \cong  d\{\text{degree 1 line bundles on } S\times_{B\PGL_2}[\mathbb{P}k^{\oplus 2}/\PGL_2]\}.
\end{align*}

For the first isomorphism, note that $S\times_{B\PGL_2} [\PP \Sym^d(k^{\oplus 2})/\PGL_2]= \Sym^d_{\PP_S/S}$, where $\PP_S=S\times_{B\PGL_2} [\PP k^{\oplus 2}/\PGL_2]$. Thus $[\ms G_d]=d[\ms G]=[d]\in \Z/2\Z$, where $\ms G$ is as in the previous example. For $d$ even, $\ms G_d$ is a trivial gerbe, corresponding to the trivial character of $\mu_2$, and for $d$ odd, $\ms G_d$ is a nontrivial gerbe, corresponding to the standard character of $\mu_2$. This is equivalent to the statement that a genus 0 curve may not have a divisor of odd degree, but it will always have a divisor of even degree, equal to some power of the canonical divisor. 
\end{example}

We conclude this section by recording a theorem of Gabber which also applies to algebraic stacks. While the proof is essentially the same as previously given, we could not find a reference for the statement generalized to stacks, so we include it here and give a sketch of its proof. 

\begin{theorem} \label{thm: Gabber}(\cite[Theorem 2, page 193]{Gabber}, extended to stacks)
    Let $S$ be a stack, $\pi: \ms X\to S$ a Brauer-Severi stack. The following sequence is exact,
    \[\begin{tikzcd}
H^0_{\et}(S,\Z)\arrow{r} & H^2(S, \G_m)_\tors \arrow{r}{\pi^*} &
H^2(\ms X,\G_m)_\tors \arrow{r} &  0,
\end{tikzcd}\]
    where the first maps sends $1\mapsto -\delta([\ms X])$. 
\end{theorem} 
\begin{proof}
    Gabber uses the Leray spectral sequence $E_2^{p,q}=H^p(S, R^q\pi_*\G_m)\Rightarrow H^{p+q}(X,\G_m)$ to obtain the exact sequence in the theorem. 
    The calculations that $\pi_*\G_m=\G_m$, $R^1\pi_*\G_m=\Z$ and $R^2\pi_*\G_m=0$ in Gabber's proof apply to Brauer-Severi stacks, because these are computed \'etale locally. 
\end{proof}

\section{The Moduli Spaces $\ms M_A$} \label{section: The moduli spaces}

Let $A=(g, N, d_1,\ldots,d_{N-1})$ be an $N+1$-tuple such that $N|d:=\sum id_i$ and 
\[\widetilde{g}:=\dfrac{1}{2}\left(N(2g-2)+\sum_i d_i(N-\gcd(i,N))+2\right)\]
is an integer greater than 1. It is then a \textit{$\widetilde{g}$-admissible datum} for some $\widetilde{g}>1$, as defined in the introduction. Define the stack $\ms M_A$ to be the stack with objects 
\[\ms M_A(S)=\{(f:C\to S, D_1, \ldots D_{N-1}, \ms L, \phi)\}\]
where $f:C\to S$ is a smooth genus $g$ curve, the $D_i$ are disjoint divisors of degree $d_i$, sections of $(\Sym_{C/S}^{d_i} \setminus \Delta_{d_i})\to S$, where $\Delta_{d_i}$ is the big diagonal, $\ms L$ is a line bundle on $C$ and $\phi: \ms L^{\otimes N}\to \ms O_{C} (\sum iD_i)$ is an isomorphism.

A morphism 
$$(\sigma, \tau): (f:C\to S, D_i, \ms L, \phi)\to (f':C'\to S, D_i', \ms L', \phi')$$
is a pair of isomorphisms $\sigma: C\to C'$ and $\tau: \sigma^*\ms L'\to \ms L$ such that $\sigma^*(D_i')=D_i$ and the following diagram commutes:

\[\begin{tikzcd}
\sigma^*(\ms L'^{\otimes N})\arrow{rr}{\sigma^*(\phi')}\arrow{d}{\tau^{\otimes N}} && \sigma^*\ms O_{C'} (\sum iD_i') \arrow{d}\\
\ms L^{\otimes N} \arrow{rr}{\phi} &  &\ms O_C(\sum iD_i).
\end{tikzcd}\]

Following \cite{Pardini1991}, we roughly explain the correspondence between $\mu_N$-covers $\pi:X\to C$ and objects $(C, D_i,\ms L, \phi)$ of $\ms M_A$. 

For any abelian cover $\pi:X\to Y$ of smooth complete varieties with group $G$, let $D$ be the locus of $Y$ where $\pi$ is ramified. Then $R:=\pi^{-1}(D)$ is the locus on $X$ whose points are stabilized by a subgroup of $G$ not equal to the identity. For a closed point $T$ of $R$, let $H\leq G$ be its stabilizer. Let $\ms O_{X,T}$ be its local ring with maximal ideal $\mathfrak m$. Since $H$ fixes $T$, we get a faithful representation $\psi$ of $H$ on the cotangent space $\mathfrak m/\mathfrak m^2$. Now we can split $D$ into disjoint components $D_{H,\psi}$ according to where $\pi^{-1}(D_{H,\psi})$ has stabilizer $H$ and representation $\psi$. Since $\pi$ is Galois, all points in the same fiber of $\pi$ will have the same stabilizer group $H$ and representation $\psi$. 

In the case where $X$ and $Y$ are curves, $\mathfrak m/\mathfrak m^2$ has dimension 1. Thus, any representation of $H$ is a character of $H$. Moreover, if additionally $G$ is cyclic of order $N$, we have a bijection 
$$\{1,\ldots, N-1\}\to \{(H,\psi): \langle 1\rangle \neq H \leq G=\mu_N, \psi\in \widehat{H}\}$$
given by $m\mapsto \left(\langle \zeta_N^m\rangle \subseteq \mu_N, \psi: \zeta_N^m\mapsto \zeta_{N}^{\gcd(m,N)}\right)$. Because of this, we can label the divisors $D_1, \ldots D_{N-1}$. Moreover, the pushforward $\pi_*\ms O_X$ of the structure sheaf can be broken into eigensheaves $\pi_*\ms O_X=\oplus_{i=0}^{N-1}\ms L_i^{-1}$, where the $\ms L_i^{-1}$ are line bundles such that $\mu_N$ acts on $\ms L_i^{-1}$ by $\zeta_N^i$, and we have an isomorphism
$\ms L_1^{\otimes N} \cong \ms O_Y(\sum iD_i).$ 

Conversely, given $(C, D_1, \ldots, D_{N-1}, \ms L, \phi)$, the cover $X$ is given by $X=\uSpec_C(\ms O_C\oplus \ms L_1^{-1}\oplus\cdots \oplus \ms L_{N-1}^{-1})$, where $\ms L_1=\ms L$ and the remaining $\ms L_i$ are determined by a formula involving the divisors, and the algebra structure of $\ms O_C\oplus \ms L_1^{-1}\oplus\cdots \oplus \ms L_{N-1}^{-1}$ is determined by $\phi$. This is worked out in detail for $\mu_N$-covers of curves in \cite{Pagani_2013}, using the general theory in \cite{Pardini1991}.

Given a $\widetilde{g}$-admissible datum $(g, N, d_i,\ldots,d_{N-1}),$ the condition for whether a corresponding covering curve is connected is as follows. Let $k$ be the greatest common divisor of $N$ and all of the $i$ such that $d_i\neq 0$. Then the covering curve is connected if and only if $\ms L^{\otimes N/k} \otimes \ms O_C(-\sum \frac{i}{k} D_i)$ has exactly order $k$ in the Picard group of $C$ (\cite[Remark 2.11]{Pagani_2013}). Note that if $C$ has genus 0, then this line bundle always has order 1, and thus when $k=1$, all covers are connected and $\ms M_A'=\ms M_A$, and when $k>1$, all covers are disconnected and $\ms M_A'=\emptyset$.

Thus, to study the $\ms M_A'$ with $g=0$, it suffices to study the $\ms M_A$ with $k=1$, where $k$ is defined as above. However, this condition on $k$ does not affect our study of the $\ms M_A$, so we ignore it.

Let $\overline{\ms M_A}$ be the stack with objects 
\[\overline{\ms M_A}(S)=\{(f:C\to S, D_1, \ldots D_{N-1}, \ms L, \phi)\}\]
where $f:C\to S$ is a smooth genus $g$ curve, the $D_i$ are divisors of degree $d_i$, sections of $\Sym_{C/S}^{d_i} \to S$, $\ms L$ is a line bundle on $C$ and $\phi: \ms L^{\otimes N}\to \ms O_{C} (\sum iD_i)$ is an isomorphism. That is, the divisors are no longer required to be disjoint. Morphisms are defined to be the same as for $\ms M_A$. 

We next define a stack $\overline{\ms N_A}$, and show that it is the $\mu_N$-rigidification of $\overline{\ms M_A}$. First we include a bit of background about the Picard stack and Picard scheme. 

\subsection{Picard stack and Picard scheme} 
Let $S$ be a scheme and let $f:X\to S$ be flat and proper. The Picard stack (\cite[Appendix]{Artin1974}) is the stack $\ms Pic_{X/S}$ with objects pairs $(T\to S,\ms L),$ where $T$ is an $S$-scheme and $\ms L$ is a line bundle on $X_T=X\times_S T$. Morphisms are pairs $(t,\sigma):(T\to S, \ms L)\to (T'\to S, \ms L')$ where $t:T\to T'$ is a morphism of $S$-schemes and $\sigma:t_X^*\ms L'\to \ms L$ is an isomorphism of line bundles, where $t_X: X_T\to X_{T'}$. 

Let $g:X\to X'$ be an $S$-morphism, and  $g_T:X_T\to X_T'$. Define 
\[g^*: \ms Pic_{X'/S}\to \ms Pic_{X/S}\]
by $(T\to S, \ms L)\mapsto(T\to S, g_T^*\ms L)$. 

We now define the relative Picard functor and Picard scheme, following \cite{kleiman2005picardscheme}. Let $f:X\to S$ be separated and of finite type.

The relative Picard functor is 
$$\Pic_{X/S}^{\pre}:(\Sch/S)\to \mathrm{Groups}$$
\[(T\to S) \mapsto \Pic(X_T)/\Pic(T).\]
In some cases, this is representable, but in general it is not. Let $\Pic_{X/S}$ denote the fppf sheafification. 

Then an object of $\Pic_{X/S}(T\to S)$ is represented by a pair $(T'\to T, \ms L')$ of an fppf covering $T'\to T$ and a line bundle $\ms L'$ on $X_{T'}$ such that there exists an fppf covering $T''\to T'\times_T T'$ where the pullbacks of $\ms L'$ to $T''$ under the two projections are isomorphic. We will call such objects {\it isomorphism classes of line bundles on $X_T$}. Two representations $(T_1'\to T, \ms L_1')$ and $(T_2'\to T, \ms L_2')$ are equivalent if there exists an fppf covering $T'\to T_1'\times_T T_2'$ such that $\ms L_1'$ and $\ms L_2'$ become isomorphic when pulled back to $X_{T'}$. 

Let $g:X\to X'$ as above, and define 
\[g^*:\Pic_{X'/S}\to \Pic_{X/S}\]
by $(T'\to T, \ms L')\mapsto (T'\to T, g_{T'}^*\ms L')$.

Now suppose that $f_*\ms O_X=\ms O_S$ and $f$ is flat, proper, and  cohomologically flat in dimension 0, that is, the formation of $f_*\ms O_X$ commutes with base change. When $f_*\ms O_X\cong \ms O_S$ holds universally and $f$ has a section \'etale locally, the \'etale sheafification and the fppf sheafification agree. In particular, when $f$ is smooth, $f$ has an \'etale local section. We have a morphism $[-]:\ms Pic_{X/S}\to \Pic_{X/S}$ sending a line bundle $\ms L$ on $X_T$ to the pair $(T\to T,\ms L)$, which we denote by $[\ms L]$.

\begin{lemma}
    Let $f: X\to S$ be flat, proper, and such that $f_*\ms O_X=\ms O_S$ holds universally. Then $[-]:\ms Pic_{X/S}\to \Pic_{X/S}$ realizes the Picard stack as a $\G_m$-gerbe over the Picard scheme.
\end{lemma}
\begin{proof}
    By descent for $\ms Pic_{X/S}$, for any fppf covering $T'\to T$, the category $\ms Pic_{X/S}(T)$ is equivalent to $\ms Pic_{X/S}(T'\to T)$, which has objects pairs $(\ms L', \sigma:\pr_1^*\ms L'\cong \pr_2^*\ms L')$, where $\ms L'$ is a line bundle on $X_{T'}$ and $\sigma$ is an isomorphism of line bundles on $X_{T'\times_T T'}$ which satisfies the cocycle condition. 

    Given $(T'\to T, \ms L')\in \Pic_{X/S}(T)$, we have some fppf cover $T''\to T'\times_T T'$ such that the two pullbacks of $\ms L'$ to $T''$ are isomorphic. Let $(T'\times_T T''\to T'', \ms L'')$ denote the pullback of $(T'\to T, \ms L')$ to $T''$. The two pullbacks of $\ms L''$ to $(T'\times_T T'')\times_{T''} (T'\times_T T'')$ agree, so descent gives a line bundle $\ms L$ on $X_{T''}$. Given $\ms L_1, \ms L_2\in \ms Pic_{X/S}(T)$ such that $(T\to T, \ms L_1)$ and $(T\to T, \ms L_2)$ are equivalent in $\Pic_{X/S}(T)$, there is an fppf cover $T'\to T$ such that $\ms L_1|_{T'}\cong \ms L_2|_{T'}$. Since $f_*\ms O_X=\ms O_S$ holds universally, $H^0(X_T,\ms O_{X_T}^\times)=\ms O_T^\times$, and automorphisms of $\ms L\in \ms Pic_{X/S}(T)$ over $[\ms L]$ are given by $\G_{m,T}$.  
\end{proof}

Let $\ms Pic_{X/S}^d$ to be the open and closed substack of $\ms Pic_{X/S}$ consisting of line bundles of relative degree $d$, and similarly let $\Pic_{X/S}^d$ be the open and closed subscheme of $\Pic_{X/S}$ consisting of isomorphism classes of line bundles of relative degree $d$. We again have a $\G_m$-gerbe $[-]:\ms Pic_{X/S}^d\to \Pic_{X/S}^d$.

In what follows, we will denote elements $(T'\to T, \ms L')$ of $\Pic_{X/S}(T)$ by a single symbol, $\ell$. By $\ell^N$, we mean the element $(T'\to T, \ms L'^{\otimes N})$.

\subsection{Rigidification of $\overline{\ms M_A}$} Define the stack $\overline{\ms N_A}$ to be the stack with objects
\[\overline{\ms N_A}(S)=\left\{\left(f:C\to S, D_1,\ldots, D_{N-1},\ell\right)\middle| \ell^N=\left[\ms O_C\left(\sum_i iD_i\right)\right]\right\}\]
where $f: C\to S$ and $D_i$ are as in the definition of $\overline{\ms M_A}$, and $\ell\in \Pic^{d/N}_{C/S}(S)$. 

A morphism 
$$\sigma: (f:C\to S, D_i, \ell)\to (f':C'\to S, D_i', \ell' )$$
is an isomorphism $\sigma: C\to C'$ of $S$-schemes such that $\sigma^*(D_i')=D_i$ and $\sigma^*\ell'=\ell$.

We have a natural map $\overline{\ms M_A}\to \overline{\ms N_A}$ given by 
\[(f:C\to S, D_i, \ms L, \phi)\mapsto (f:C\to S, D_i, [\ms L]).\]

\begin{proposition}
    $\overline{\ms M_A}\to \overline{\ms N_A}$ is a $\mu_N$-gerbe and thus $\overline{\ms N_A}$ is the $\mu_N$-rigidification of $\overline{\ms M_A}$. 
\end{proposition}
\begin{proof}
    Let $(f:C\to S, D_i, \ell)\in \overline{\ms N_A}(S)$. Then $\ell$ is a pair $(S'\to S, \ms L')$ such that there is an fppf cover $S_1\to S'\times_S S'$ where the pullbacks of $\ms L'$ agree. Moreover, that $\ell^N=[\ms O_C(\sum_i iD_i)]$ means there exists an fppf cover $S_2\to S'\times_S S=S'$ where the pullbacks of $\ms L'^{\otimes N}$ and $\ms O_C(\sum_i iD_i)$ agree. Taking a mutual covering of $S_1$ and $S_2$, say $S''=S_1\times_S S_2$ we get a line bundle $\ms L$ on $C_{S''}$, which is the pullback of $\ms L'$, and an isomorphism $\phi: \ms L^{\otimes N}\to \ms O_{C_{S''}}(\sum_i iD_i)$. 

    Let $(f:C\to S, D_i, \ms L_1, \phi_1)$ and $(f:C\to S, D_i, \ms L_2, \phi_2)$ be two elements of $\overline{\ms M_A}(S)$ over $(f:C\to S, D_i, [\ms L_1]=[\ms L_2])\in \overline{\ms N_A}(S)$. Then there exists an fppf cover $S'\to S$ such that $\ms L_1|_{S'}\cong \ms L_2|_{S'}$. Choose an isomorphism, say $\phi$. Then $\phi_2\circ \phi$ agrees with $\phi_1$ up to multiplication by some $s\in \G_m(S')$, and fppf locally, we may take the $N$th root of $s$, and compose with $\phi$, so that fppf locally, our elements of $\overline{\ms M_A}(S)$ are isomorphic.

    Finally, automorphisms of $(f:C\to S, D_i, \ms L, \phi)$ over $(f:C\to S, D_i, [\ms L])$ are given by $\mu_{N, S}$, since they consist of pairs $(\id, \tau)$ such that $\tau:\ms L\cong \ms L$ and $\tau^{\otimes N}=\id$.
\end{proof}

We also define $\ms N_A$ to be the stack with objects and morphisms of $\overline{\ms N_A}$ having disjoint divisors, as in $\ms M_A$. Then $\ms M_A\to \ms N_A$ is also a $\mu_N$-gerbe, and $\overline{\ms M_A}\times_{\overline{\ms N_A}} \ms N_A\cong \ms M_A$.

Let $\ms C_g$ be the universal curve of genus $g$ and consider the stack $\Pic^{d/N}_{\ms C_g/\ms M_g}$, with objects 
\[\Pic^{d/N}_{\ms C_g/\ms M_g}(S)=\left\{\left(f:C\to S, \ell\in \Pic^{d/N}_{C/S}(S)\right)\right\},\]
and morphisms $\sigma:(f:C\to S,\ell)\to (f:C'\to S, \ell')$ given by isomorphisms $\sigma: C\to C'$ such that $\sigma^*\ell'=\ell$. 

We have a natural map $\overline{\ms N_A}\to \Pic^{d/N}_{\ms C_g/\ms M_g}$ given by $(f:C\to S, D_i, \ell)\mapsto (f:C\to S, \ell)$. Note also that 
\[\overline{\ms N_A}=\left(\Sym^{d_1}_{\ms C_g/\ms M_g} \times_{\ms M_g} \cdots \times_{\ms M_g} \Sym^{d_{N-1}}_{\ms C_g/\ms M_g}\right)\times_{\Pic_{\ms C_g/\ms M_g}^d} \Pic^{d/N}_{\ms C_g/\ms M_g},\]
where the fiber product is taken over the map that sends $(f:C\to S, \{D_i\})$ to $(f:C\to S,[\ms O_C(\sum_iiD_i)])$, and the map that sends $(f:C\to S, \ell)$ to $(f:C\to S, \ell^N)$.
The map $\overline{\ms N_A}\to \Pic^{d/N}_{\ms C_g/\ms M_g}$ is projection onto the second factor. 

\begin{proposition} \label{prop: genus 0 simplifications}
    When $g=0$, we can make the following simplifications. 
    \begin{enumerate}
        \item $\ms M_0=B\PGL_2$.
        \item $\ms C_0=[\PP V/\PGL_2]$ is the universal curve of genus 0, where $V=k^{\oplus 2}$ and $\PGL_2$ acts on $\PP V$ via the standard representation of $\SL_2$ on $V$. 
        \item $\Pic^d_{\ms C_0/\ms M_0}\cong B\PGL_2$ for all $d$.
        \item $\Sym^d_{\ms C_0/\ms M_0}=[\PP \Sym^d V/\PGL_2]$ is the moduli of degree $d$ divisors on genus 0 curves.
        \item $\Sym^{d_1}_{\ms C_0/\ms M_0}\times_{B\PGL_2} \cdots \times_{B\PGL_2} \Sym^{d_{N-1}}_{\ms C_0/\ms M_0}=[(\PP \Sym^{d_1}V\times \cdots \times \PP \Sym^{d_{N-1}}V)/\PGL_2]$.
    \end{enumerate}
\end{proposition}
\begin{proof}
    (1) is Example \ref{ex: M_0=BPGL_2}, a special case of Proposition \ref{prop: br-sev eq pgl}. 
    
    For (2), let $V$ be the standard representation of $\SL_2$, so $V=k^{\oplus 2}$, and $\SL_2$ acts via 
    $\begin{pmatrix}
    A & B\\
    C & D
    \end{pmatrix}\begin{pmatrix}
        x \\ y
    \end{pmatrix}= \begin{pmatrix}
        Ax+By \\ Cx+Dy
    \end{pmatrix}$, as in Example \ref{example:PGL_2}. Then $\mu_2$ acts by the standard character, and the action descends to an action of $\PGL_2$ on $\PP V$. Let $f:X\to S$ be a genus 0 curve, and $S\to B\PGL_2$ be the torsor $I_X:(T\to S)\mapsto\{\sigma: X\times_S T\cong \PP V_T=\PP^1_T\} $. Due to the commutativity of the diagram
    \[\begin{tikzcd}
        \PP V\times I_X \arrow{rr} \arrow{dr}\arrow{dd} && \PP V \arrow{dr}\arrow{dd} & \ \\ & I_X \arrow{rr}\arrow{dd} && \Spec k \arrow{dd} \\ (\PP V\times I_X)/\PGL_2 \arrow{rr}\arrow{dr} && {[\PP V/\PGL_2]} \arrow{dr} & \\ \ & S \arrow{rr} && B\PGL_2,
    \end{tikzcd}\]
    we would like to show that $(\PP V\times I_X)/\PGL_2\cong X$. Indeed an isomorphism is given by sending $(x,\sigma)\in \PP V\times I_X$ to $\sigma^{-1}(x)$, which descends to an isomorphism on the quotient. 

    For (3), we have $\Pic_{\ms C_0/\ms M_0}\cong \Z\times \ms M_0$, and for each $d$, $\Pic^d_{\ms C_0/\ms M_0}\cong \ms M_0$, represented by the class of $\ms O(d)$ on $\PP V$. 

    For (4) and (5), we use the fact that for two $\SL_2$ representations $V_1$ and $V_2$ that descend to actions of $\PGL_2$ on $\PP V_1$ and $\PP V_2$, we have $[(\PP V_1 \times \PP V_2)/\PGL_2]\cong [\PP V_1/\PGL_2]\times_{B\PGL_2} [\PP V_2/\PGL_2]$. One can see this by taking the $\PGL_2$ quotient of the cartesian diagram of $\PP V_1\times \PP V_2$. Assuming statement (4), statement (5) follows. For (4), we may view $\PP\Sym^dV$ as the space of polynomials of degree $d$ in 2 variables up to multiplication by a scalar, and the zero locus of this polynomial is a degree $d$ divisor on $\PP V$. Thus, $\PP V\cong \PP \Sym^1V$ and $\PP\Sym^dV=(\PP V\times \cdots \times \PP V)/S_d$. Let $f:X\to S$ be a genus 0 curve. Using a similar argument as in the proof of (2), we see that 
    \[((\PP V\times \cdots \times \PP V)/S_d \times I_X)/\PGL_2=(X\times_S \cdots \times_S X)/S_d,\]
    which shows that $[\PP \Sym^dV/\PGL_2]$ is the moduli of degree $d$ divisors on genus 0 curves. Alternatively,
    \[\Sym^d_{\ms C_0/\ms M_0}=(\ms C_0\times_{\ms M_0} \cdots \times_{\ms M_0} \ms C_0)/S_d=[(\PP V \times \cdots \times \PP V)/\PGL_2]/S_d\]
    \[=[((\PP V\times \cdots \times \PP V)/S_d)/\PGL_2]=[\PP \Sym^dV/\PGL_2]. \qedhere\]
    
\end{proof}

From the proposition, we can conclude that when $g=0$,
\[\overline{\ms N_A}=[(\PP \Sym^{d_1}V\times \cdots \times \PP \Sym^{d_{N-1}}V)/\PGL_2]\to \Pic^{d/N}_{\ms C_0/\ms M_0}=B\PGL_2\]
is a composition of Brauer-Severi stacks. Indeed, write 
\[\ms X_i:=[\PP \Sym^{d_1} V/\PGL_2]\times_{B\PGL_2}\cdots \times_{B\PGL_2} [\PP \Sym^{d_{i}}V/\PGL_2],\]
and we factor $\gamma$ by
\begin{equation} \label{eq: factor gamma}
\begin{tikzcd}
    \gamma: \overline{\ms N_A}=\ms X_{N-1}\arrow{r}{\gamma_{N-1}}  &\cdots \arrow{r}{\gamma_2} & \ms X_1\arrow{r}{\gamma_1} & B\PGL_2,
\end{tikzcd}
\end{equation}
where each $\gamma_i$ is a projection onto the first $i-1$ factors. Then $\ms X_i=\ms X_{i-1}\times_{B\PGL_2} [\PP \Sym^{d_i}V/\PGL_2],$ so $\ms X_i\to \ms X_{i-1}$ is a Brauer-Severi stack, given by the pullback of $[\PP \Sym^{d_i}V/\PGL_2]\to B\PGL_2$.

\section{Proof of Theorems \ref{thm:N_A} and \ref{thm:M_A}} \label{sec: main proof}

Let $A=(0, N, d_1, \ldots, d_{N-1})$ be as in Section \ref{section: The moduli spaces}. From the factorization \ref{eq: factor gamma} of $\gamma:\overline{ \ms N_A}\to B\PGL_2$ as a composition of Brauer-Severi stacks, we get a factorization of $\gamma^*:H^2(B\PGL_2, \G_m)\to H^2(\overline{\ms N_A}, \G_m)$ as a composition of the $\gamma_i^*: H^2(\ms X_{i-1},\G_m)\to H^2(\ms X_i, \G_m)$, and by Gabber's Theorem (\ref{thm: Gabber}), these are surjections such that $\delta([\ms X_i])\mapsto 0$. 

Recall that $\ms X_i=\ms X_{i-1}\times_{B\PGL_2} [\PP \Sym^{d_i}V/\PGL_2]$ is the Brauer-Severi stack given by the pullback of $[\PP \Sym^{d_i}V/\PGL_2]\to B\PGL_2$ to $\ms X_i$. We have 
\[\delta([[\PP \Sym^{d_i}V/\PGL_2]])\in H^2(B\PGL_2,\G_m)=\Z/2\Z\]
given by $[d_i]\in \Z/2\Z$ by Example \ref{example:Sym^d}, so $\delta([\ms X_i])\in H^2(\ms X_{i-1},\G_m)$ is just the image of $[d_i]$ under 
\[H^2(B\PGL_2,\G_m)\twoheadrightarrow H^2(\ms X_{i-1},\G_m).\]

Thus under $H^2(BG,\G_m)\twoheadrightarrow H^2(\ms X_i,\G_m)$, $[d_i]\mapsto 0$, and under $\gamma^*: H^2(BPGL_2, \G_m)\twoheadrightarrow H^2(\overline{\ms N_A}, \G_m)$ and $[d_i]\mapsto 0$ for all $i=1,\ldots, N-1$. Then if all $d_i$ are even, $H^2(\overline{\ms N_A},\G_m)=\Z/2\Z$, and if any $d_i$ is odd, $H^2(\overline{\ms N_A},\G_m)=0$. 

This concludes the proof of Theorem \ref{thm:N_A}.
\qed

Now, since $H^2(\overline{\ms N_A},\G_m)$ is either $0$ or $\Z/2\Z$, depending on the parities of the $d_i$, the Brauer class of $\overline{\ms M_A}$ is either trivial or nontrivial of order 2. We now compute precisely what the Brauer class of $\overline{\ms M_A}$, and thus also $\ms M_A$, is. 

We work with the following diagram
\[\begin{tikzcd}
    \mathbb{P}_{\overline{\ms N_A}} \arrow{r}{\gamma'} \arrow{d}[swap]{\pi'} & \mathbb{P} \arrow{d}{\pi} \\ \overline{\ms N_A} \arrow{r}{\gamma} &B\PGL_2,
\end{tikzcd}\]
where $\PP:=[\PP V/\PGL_2]$ is the universal curve of genus 0, as in Example \ref{example:PGL_2}.

Consider the $\G_m$-gerbe over $B\PGL_2$,
\[\ms G_{d/N}(S\to B\PGL_2)=\{ \text{degree } d/N \text{ line bundles on } S\times_{B\PGL_2} \mathbb{P} \}.\]
 
We have $[\ms G_{d/N}]=[d/N]\in H^2(B\PGL_2, \G_m)=\Z/2\Z.$ 

We will show that $[\overline{\ms M_A}]^{\wedge \G_m}=[\gamma^*\ms G_{d/N}]$, where $[\overline{\ms M_A}]^{\wedge \G_m}$ is the Brauer class of $\overline{\ms M_A}$, that is, the image of $[\overline{\ms M_A}]$ under the map $H^2(\overline{\ms N_A}, \mu_N)\to H^2(\overline{\ms N_A},\G_m)$ coming from the inclusion $\mu_N\to \G_m$, and $\gamma^*\ms G_{d/N}$ is the $\G_m$-gerbe over $\overline{\ms N_A}$ given by the cartesian product $\ms G_{d/N}\times_{B\PGL_2} \overline{\ms N_A}$. In particular,
\[\gamma^*\ms G_{d/N}(S\to \overline{\ms N_A})=\{ \text{degree } d/N \text{ line bundles on } S\times_{B\PGL_2} \mathbb{P} \}.\]

We define a morphism of stacks $\overline{\ms M_A}\to \gamma^*\ms G_{d/N}$ by 
\[(f:C\to S, D_i, \ms L, \phi)\mapsto ((f:C\to S, D_i, [\ms L]),\ms L).\]
We can identify $S\times_{B\PGL_2}\PP$ with the curve $C$, so to give an object of $\gamma^*\ms G_{d/N}$, we must give a degree $d/N$ line bundle on $C$, namely $\ms L$. 
The map on morphisms is the identity map, thus the map on stabilizers is the inclusion $\mu_N\to \G_m$, so indeed $[\overline{\ms M_A}]^{\wedge \G_m}=[\gamma^*\ms G_{d/N}]$. 

If $d/N$ is even, then $\ms G_{d/N}$ is the trivial gerbe over $B\PGL_2$, and if $d/N$ is odd, then it is nontrivial, isomorphic to $B\GL_2$. Together with the calculation of $H^2(\overline{\ms N_A},\G_m)$, the result follows. Finally, consider the open restriction map
\[\mathrm{res}: H^2(\overline{\ms N_A},\G_m)\to H^2(\ms N_A,\G_m),\]
which sends $[\overline{\ms M_A}^{\wedge \G_m}]\mapsto [\ms M_A^{\wedge \G_m}]$. Due to the injectivity of the restriction map on Brauer groups, we get the same result for $[\ms M_A^{\wedge \G_m}].$
\qed

\begin{remark}
    In the case where $d$ is even, we can actually write down the class of $\overline{\ms M_A}$ as a $\mu_N$-gerbe over $\overline{\ms N_A}$, rather than just the class of its pushout to $\G_m$. However, as remarked in the introduction, since the restriction map is not injective on $H^2(\mu_N)$, this is not sufficient to compute the $\mu_N$-class of the open substack $\ms M_A$, whose moduli we are interested in.
 
    Our gerbe of interest, $\overline{\ms M_A}\to \overline{\ms N_A}$ can be rewritten as follows,
    \[\overline{\ms M_A}(f: S\to \overline{\ms N_A})=\{\ms L, \phi: \ms L^{\otimes N} \cong f'^* \ms O(\ms D)\},\]
    where $\ms L$ is a line bundle on $\PP_S$ (which we have identified with $C$), $\ms D=\sum_i\ms D_i$ is the sum of the universal divisors on $\PP_{\ms N_A}$,  and $f': \PP_S\to \PP_{\ms N_A}$.

    Consider the $\G_m$-gerbes over $B\PGL_2$, using notation from Corollary \ref{cor: V_1,...V_r},
    \[\ms G_d(S\to B\PGL_2)=\{ \ms O(d) \text{ on } S\times_{B\PGL_2} \mathbb{P} \},\]
    \[\ms G_{d_1, \ldots, d_{N-1}}^{1,\ldots, N-1}(S\to B\PGL_2)=\left\{ \ms O(1,\ldots, N-1) \text{ on } S\times_{B\PGL_2} \overline{\ms N_A} \right\}.\]
    We have 
    \[[\ms G_d]=[d]\in H^2(B\PGL_2, \G_m)=\Z/2\Z,\]
    \[\left[\ms G_{d_1, \ldots, d_{N-1}}^{1,\ldots, N-1}\right]=\left[\sum id_i\right]=[d]\in H^2(B\PGL_2, \G_m)=\Z/2\Z,\]
    by Corollary \ref{cor: V_1,...V_r} and Example \ref{example:Sym^d}. From this, we can conclude that $\ms G_d\cong \ms G_{d_1, \ldots, d_{N-1}}^{1,\ldots, N-1}$, and the isomorphism comes from the fact that, whenever such line bundles exist, 
    \[\ms O(\ms D)\cong \gamma'^* \ms O(d) \otimes \pi'^*\ms O(1,\ldots, N-1).\]
    The existence of $\ms O(d)$ on $\mathbb{P}$ is equivalent to the existence of $\ms O(1,\ldots, N-1)$ on $\overline{\ms N_A}$, and both exist after pulling back to $\ms G_d$. If $d$ is even, the class of $\ms G_d$ is trivial, meaning that $\ms O(d)$ exists on $\mathbb{P}$ and $\ms O(1,\ldots, N-1)$ exists on $\overline{\ms N_A}$.

    Thus, if $d$ is even, we have $\overline{\ms M_A}\cong \ms M'\otimes \ms M''$, where
    \[\ms M'(f:S\to \overline{\ms N_A})=\{ \ms L', \phi': \ms L'^{\otimes N}\cong f'^* \pi'^* \ms O(1,\ldots, N-1)\},\]
    \[\ms M''(f:S\to \overline{\ms N_A})=\{ \ms L'', \phi'': \ms L''^{\otimes N} \cong f'^* \gamma'^* \ms O(d) \}.\]

    First, note that $\ms M'$ is isomorphic to the $\mu_N$-gerbe of $N$th roots of $\ms O(1,\ldots, N-1)$ on $\overline{\ms N_A}$, 
    \[\sqrt[N]{\ms O(1,\ldots, N-1)}\left(f:S\to \overline{\ms N_A}\right)=\{L,\phi: L^{\otimes N}\cong f^*\ms O(1,\ldots, N-1) \text{ on } S\}.\]
    A map is given by sending $(L,\phi)$ to $(\pi^*L, \pi^*\phi)$. From the exact sequence
    $1\to \mu_N\to \G_m \to \G_m \to 1$, we get
    \[H^1(\overline{\ms N_A},\G_m)\to H^2(\overline{\ms N_A},\mu_N)\to H^2(\overline{\ms N_A},\G_m),\]
    and the first map thus sends the line bundle $\ms O(1,\ldots, N-1)$ to $\ms M'$, so under the second map, $\ms M'\mapsto 0.$ Thus, $[\ms M']^{\wedge \G_m} =[B\G_m] \in H^2(\overline{\ms N_A}, \G_m)$, and thus $[\ms M'']^{\wedge \G_m} =[\overline{\ms M_A}]^{\wedge \G_m}$. 

    Next, the same exact sequence gives 
    \[0= H^1(B\PGL_2, \G_m)\to  H^2(B\PGL_2, \mu_N)\to   H^2(B\PGL_2, \G_m)\to H^2(B\PGL_2, \G_m)\]
    with $H^2(B\PGL_2, \G_m)=\Z/2\Z$. The rightmost map sends $[\ms G_{d/N}]\mapsto [\ms G_d]=0$, so the class of $\ms G_{d/N}$ comes from the pushout of a $\mu_N$-gerbe on $B\PGL_2$, which can be described as follows,
    \[\ms H(f:S\to B\PGL_2)=\{ \ms L, \phi: \ms L^{\otimes N} \cong f'^*\ms O(d) \text{ on } S\times_{B\PGL_2} \mathbb{P} \}.\]

    Thus we see that $\ms M''=\gamma^*\ms H$, and  
    $$\overline{\ms M_A}\cong \ms M'\otimes \ms M''\cong \sqrt[N]{\ms O(1,\ldots, N-1)} \otimes \gamma^* \ms H.$$

    Moreover, this confirms that
    \[[\overline{\ms M_A}]^{\wedge \G_m} =[\ms M'']^{\wedge \G_m}=[\gamma^*\ms H]^{\wedge \G_m} =\gamma^* [\ms G_{d/N}].\]
\end{remark}

\bibliographystyle{alpha}
\bibliography{references}

\end{document}